\documentclass[12pt,a4paper]{amsart}
\usepackage{amssymb,graphicx,tikz-cd,mathtools}
\def\onto{\twoheadrightarrow}
\def\Z{\mathbb{Z}}
\def\ol{\overline}
\def\wh{\widehat}
\def\<{\langle}
\def\>{\rangle}
\def\G{\Gamma}
\def\D{\Delta}
\def\a{\alpha}
\def\b{\beta}
\def\g{\gamma}
\def\d{\delta}
\def\cS{\mathcal{S}}
\def\cT{\mathcal{T}}

\newtheorem{thmA}{Theorem}

\newtheorem{thm}{Theorem}[section]
\newtheorem{lem}[thm]{Lemma}
\newtheorem{cor}[thm]{Corollary}
\newtheorem{rmk}[thm]{Remark}

\catcode`\@=11
\def\serieslogo@{\relax}
\def\@setcopyright{\relax}
\catcode`\@=12

\begin{document}
	\title{ 
		Generalising Collins' Theorem }
	
	\author[Howie]{James Howie }
	\address{ James Howie\\
		Department of Mathematics and Maxwell Institute for Mathematical Sciences\\
		Heriot--Watt University\\
		Edinburgh EH14 4AS }
	\email{ j53howie@gmail.com, j.howie@hw.ac.uk}
	
	
	\author[Short]{ Hamish Short } 
	\address{ Hamish Short \\
		Institut de Mathematiques de Marseille (I2M)\\
		3 Place Victor Hugo Case 19\\
		13331 Marseille Cedex 3, France\\
		Aix Marseille Univ, CNRS, I2M}
	\email{ hamish.short@univ-amu.fr }
	
	\thanks{The first named author was supported in part by 
		Leverhulme Trust Emeritus Fellowship EM-2018-023$\backslash$9 }
	\keywords{Locally indicable, one-relator product, Magnus subgroup}
	\subjclass[2020]{Primary 20F65, 57K20. Secondary 20E06, 20F06, 57M07}
	\abstract{We generalise a result of D. J. Collins on intersections of conjugates of 
		Magnus subgroups of one-relator groups to the context of one-relator products of locally indicable groups.}}

\dedicatory{Dedicated to the memory of Donald J. Collins}
\maketitle

\section{Introduction}\label{intro}

 The basic objects of combinatorial and geometric group theory are the free groups, the next
	level of complication resulting from adding a relation to the presentation to give the class of one-relator groups.
	A natural generalisation of this is 
	the class of one-relator products, obtained by adding a single relator to a free
	product of two or more groups. Over the years many results concerning one-relator groups have been generalised to one-relator
	products, at least when the factor groups are locally indicable 
	(every non-trivial finitely generated subgroup has $\Z$ as a homomorphic image).
	A basic result on one-relator groups is Magnus' Freiheitssatz \cite{M1} which states that the subgroup generated by a subset of the generators of a one-relator group is free if at least one generator appearing in the relation is not present in the subset. 
Subgroups of this form are now known as {\em Magnus subgroups}. 
The technique introduced by Magnus in this early work is generally known as {\em Magnus induction} or {\em the Magnus hierarchy}.  
Effectively, the one-relator group can be embedded in an HNN-extension of a simpler one-relator group in which the associated subgroups are Magnus subgroups.  
It and its variants form the basis for much of the rich theory of one-relator groups that has been developed in the subsequent (almost) century. 

 We consider here one-relator products of locally indicable groups -- that is 
 groups of the form $(*_\lambda G_\lambda)/\<\<R\>\>$ where $\{G_\lambda;\lambda\in\Lambda\}$ is a family of locally indicable groups and $\<\<R\>\>$ is the normal closure in their free product of a single element $R$.   This goes back to early work of Brodski\u\i\ \cite{B80,B81} and of the authors \cite{H81,S81} and remains a productive source of new results to this day -- see for example \cite{HS}. 

	The object of this article is the generalisation to one-relator products of a result due to Don Collins concerning one-relator groups. The methods used are frequently geometric, considering finite 2--dimensional complexes naturally associated to finite group presentations.

In two articles \cite{C1,C2}, Collins gave a complete analysis of the possible intersections of Magnus subgroups and their conjugates in a one-relator group $G$ \ :

\begin{thm}[\cite{C1}]\label{Co1}
	If $M(Y),M(Z)$ are the Magnus subgroups of $G$ generated by $Y,Z\subset X$ respectively, then the intersection of $M(Y)$ and $M(Z)$ in $G$ is either the Magnus subgroup $M(Y\cap Z)$ or $M(Y\cap Z)*C$ for some cyclic group $C$.
\end{thm}

\begin{thm}[\cite{C2}]\label{Co2}
	If $g\in G$ and the intersection of $M(Y)$ and $g^{-1}M(Z)g$ in $G$ is not cyclic, then there are elements $h\in M(Z)$ and $k\in M(Y)$ such that $g=hk$ in $G$.
\end{thm}

Note that, in the conclusion of Theorem \ref{Co2} we have $h^{-1}M(Z)h=M(Z)$ and $k^{-1}M(Y)k=M(Y)$, so $M(Y)\cap g^{-1}M(Z)g=k^{-1}(M(Y)\cap M(Z))k$ and Theorem \ref{Co1} applies.

 Since Magnus subgroups are free, in Theorems \ref{Co1} and \ref{Co2} ``cyclic'' means either infinite cyclic or trivial.  Indeed in most cases the cyclic subgroup in these theorems turns out to be trivial. (For similar reasons the cyclic subgroups in other analogous results in this paper will be either infinite cyclic or trivial.)  

In the case of a one-relator group with torsion, the above results can be strengthened \cite{C1}: The second case $M(Y\cap Z)*C$ in Theorem \ref{Co1} cannot occur, and in Theorem \ref{Co2} one can replace ``cyclic'' by ``trivial''.
 
The importance of Collins' results has 
been highlighted in a number of recent advances by Linton on one-relator groups (cf. for example \cite{L1,L2} or
 the survey in \cite[Chapter 2]{LN}).  
A key issue in this work is the intersections of edge-stabilisers in the Bass-Serre tree of the HNN extension in the Magnus hierarchy -- which are of course conjugates of the Magnus subgroups of the vertex stabilisers.

In \cite{H05}, the first-named author generalised Theorem \ref{Co1} -- and also the stronger version in the torsion case -- to the situation of one-relator products of locally indicable groups.
The present paper provides analogous generalisations of both versions of Theorem \ref{Co2}.  Specifically, we prove the following two results.

\begin{thmA}\label{main}
	Let $\{G_\lambda;\lambda\in\Lambda\}$ be a collection of locally indicable groups, let $R\in \ast_{\lambda\in\Lambda} G_\lambda$ be a cyclically reduced word of free product length at least $2$, and let $G:=(\ast_{\lambda\in\Lambda} G_\lambda)/\<\<R\>\>$, where $\<\<R\>\>$ denotes the normal closure of $R$.  Let $I,J$ be subsets of $\Lambda$ and $g\in \ast_{\lambda\in\Lambda} G_\lambda$.  If the intersection of $\ast_{i\in I} G_i$ and
	$g^{-1}\cdot\left(\ast_{j\in J} G_j\right)\cdot g$
	in $G$ is not cyclic, then there are elements $h\in \ast_{j\in J} G_j$ and $k\in \ast_{i\in I} G_i$ such that $g=hk$ in $G$.
\end{thmA}

\begin{thmA}\label{maintorsion}
	Let $G$ be as in Theorem \ref{main} with $R=Z^m$ in $\ast_{\lambda\in\Lambda} G_\lambda$ for some $Z\in \ast_{\lambda\in\Lambda} G_\lambda$ and some $m>1$. If the intersection of $\ast_{i\in I} G_i$ and 
	$g^{-1}\cdot\left(\ast_{j\in J} G_j\right)\cdot g$
	in $G$ is not trivial, then there are elements $h\in \ast_{j\in J} G_j$ and $k\in \ast_{i\in I} G_i$ such that $g=hk$ in $G$.
\end{thmA}

Our proofs will make extensive use of the following special case of Theorem \ref{main} due to Brodski\u\i\ \cite[Theorem 6]{B81} (see also \cite[Theorem F]{HS}).

\begin{thm}\label{BrodskyLemma}
	Let $G$ be as in Theorem \ref{main}.  If $g\in G$ and $\lambda,\mu\in\Lambda$ are such that the intersection in $G$ of $G_\lambda$ and $g^{-1}G_\mu g$ is not cyclic, then $\lambda=\mu$ and $g\in G_\lambda$.
\end{thm}

We also require the following stronger version when the relator is a proper power.  See \cite[Corollary 3.7]{HS}.

\begin{thm}\label{BrodskyTorsion}
	Let $G$ be as in Theorem \ref{maintorsion}.  If $g\in G$ and $\lambda,\mu\in\Lambda$ are such that the intersection in $G$ of $G_\lambda$ and $g^{-1}G_\mu g$ is not trivial, then $\lambda=\mu$ and $g\in G_\lambda$.
\end{thm}

Other results on locally indicable groups from the literature that we will use include the following:

\begin{thm}[Freiheitssatz,\cite{B80},\cite{H81},\cite{S81}]\label{Freiheitssatz}
	
	Let $A,B$ be locally indicable groups and $R\in A*B$ 
	a cyclically reduced word of length at least two. 
	
	The  natural map $A\to \frac{A*B}{\<\< R\>\>}$ is injective.	
\end{thm}

\begin{thm}[Weinbaum's Theorem, \cite{H82}]\label{Weinbaum}
	Let $A,B$ be locally indicable groups and $R\in A*B$ 
	a cyclically reduced word of length at least two. 
	
	No proper subword of $R$ represents the identity element of $\frac{A*B}{\<\< R\>\>}$. 
	
\end{thm} 

\begin{thm}[\cite{H82}]\label{Gli}
	Let $A,B$ be locally indicable groups and $R\in A*B$ 
	a cyclically reduced word of length at least two which is not a proper power in $A*B$. 
	Then $\frac{A*B}{\<\< R\>\>}$ is locally indicable. 
\end{thm}
A {\em right ordering} on a group $G$ is a total order $<$ such that $(\forall ~x,y,z\in G)$ $y<z\Rightarrow yx<zx$.  
A group is {\em right orderable} if it has a right ordering.

\begin{thm}[\cite{BH}]\label{lilo}
	Every locally indicable group is right orderable.
\end{thm}

In \S \ref{prelims} 
we reformulate the main theorems in terms of 2--complexes with
two distinguished sets of edges.  
We explain 
why the proof of Theorem \ref{main} reduces to consideration of a free subgroup 
of rank $2$ in the conjugacy intersection.
Also 
we explain how pictures are to be used, 
and the essential use of a right order to select maximal and minimal edges
on each disc in the pictures is described.

In \S3 we use the methods of \S \ref{prelims} to prove the torsion version of the main theorem, Theorem \ref{maintorsion}.
 
We begin \S4 with the plan for the proof in the torsion-free case, Theorem \ref{main}, which is broken into
several stages.  We then  introduce a notion of complexity that forms the basis for an inductive proof.

We make some technical adjustments to our $2$-complex, and prove the initial step of the induction. 
The inductive step involves lifting our pictures to a suitable cover; and then splits into two cases,
depending on the nature of the lifted pictures.  The first, more straightforward, of these
 is handled in \S \ref{first}.  The other, more complicated, case is dealt with in the final three sections.  In \S \ref{restrict} we reduce to consideration of one particular conjugating element; in \S \ref{formW1} we show that annular pictures for conjugacies with that particular conjugating element have a special form; and finally in \S \ref{finalcurtain} we use that special form to complete the proof. 

 \medskip 
We are grateful to an anonymous referee for several helpful suggestions that have improved our exposition.

\section{Preliminaries}\label{prelims} 

\subsection{Reformulation and reductions} \label{refandreds}

\begin{lem}\label{red1}
Theorem \ref{main} (resp. \ref{maintorsion}) is true if and only if it is true in the case where $\Lambda=I\cup J$
\end{lem}

\begin{proof}
\noindent  The ``only if'' part is clear, so suppose that the result holds when $\Lambda=I\cup J$.
Suppose that $\Lambda \setminus (I\cup J)\neq \emptyset$ and set $D=\ast_{\lambda\in\Lambda \setminus (I\cup J)} G_\lambda$.
If $R$ contains no occurrences from $D$, then $\frac{(\ast_{j\in \Lambda} G_\lambda)}{\<\<R\>\>} = \bigl(\frac{\ast_{\lambda\in I\cup J} G_\lambda}{\<\<R\>\>}\bigr) *D$, 
and in a free product, 
elements of a factor are  conjugate if
and only if they are conjugate in that factor, and the conjugating element is an element of the same factor.

Now suppose  that  $R$ contains  occurrences in $D$.
Write $A=\ast_{i\in I\setminus J}G_i$, $B=\ast_{j\in J\setminus I}G_j$, $C=\ast_{\lambda\in I\cap J}G_\lambda$.
By the Freiheitssatz (Theorem \ref{Freiheitssatz}) we may regard $A*B*C$ as a subgroup of $G$.  The conjugacy intersection of interest  is contained in $(A*B*C)\cap g^{-1}(A*B*C)g$, and the relevant version of 
Brodski\u\i's Theorem (\ref{BrodskyLemma} or \ref{BrodskyTorsion}) 
applies to say that the conjugating element $g$ lies in the subgroup $A*B*C$. 
The result then follows from the conjugacy properties of free products as above.
\end{proof}

	So from now on we shall assume that $\Lambda = I\cup J$.

\smallskip
\noindent\underline{Reformulation} :	 We rephrase   Theorems \ref{main} and \ref{maintorsion} in terms of 2--complexes.	

Let $Y=X\cup\cS\cup\cT\cup\alpha$ be a connected 2--complex, where:
\begin{enumerate}
	\item[$\bullet$]$X$ is a 2--complex 
	(not necessarily connected),  
	\item[$\bullet$] each component of $X$  has locally indicable (possibly trivial)
	fundamental group, 
	\item[$\bullet$] $\cS,\cT$ are  disjoint, non--empty finite sets of edges,
\item[$\bullet$] the $2|\cS\cup\cT|$ endpoints of $(\cS\cup\cT)$-edges are pairwise distinct vertices of $X$,
	\item[$\bullet$] $\alpha$ is a 2-cell adjoined by identifying its boundary
$\partial\a$ with the loop in $X^{(1)}\cup\cS\cup\cT$, 
	labelled $R=Z^m$ where $Z=u_1x_1u_2x_2\dots u_kx_k$ and $m\ge 1$
(throughout the paper, $\partial\b$ will denote the boundary of a $2$-cell $\b$ in some complex, regarded as an edge path in the $1$-skeleton of the complex),
\item[$\bullet$] $Z$ does not represent a proper power in $\pi_1(X\cup\cS\cup\cT)$, 
	\item[$\bullet$]  each $u_i$ in $Z$ is $e^{\pm1}$ for some edge $e\in\cS\cup\cT$, 
	each $x_i$ is  an  edge-path of length $1$ in (some component of) $X^{(1)}$,
\item[$\bullet$] $Z$ is cyclically reduced: if $x_j$ is a  nullhomotopic loop 
 in $X$ 
then $u_{j+1}\ne u_j^{-1}$ (indices mod $k$), 
	\item[$\bullet$] every edge $e\in\cS\cup\cT$ occurs in $Z$  (in particular  $k\ge 2$).
\end{enumerate}

\smallskip\noindent
The conditions ensure that $\pi_1(Y)$ is a one-relator product
$$(F * (*_{\lambda\in\Lambda} G_\lambda ))/\<\<R\>\>$$
where $F$ is a free group (possibly trivial), 
and the $G_\lambda$ are the fundamental groups of the components of $X$. 
Let $$\ol{G}:=(F * (*_{\lambda\in\Lambda} G_\lambda ))/\<\<Z\>\>.$$
Then  
$\ol{G}$ is itself locally indicable by Theorem \ref{Gli}. 
In particular when $m=1$ we have $R=Z$ and 
$\pi_1(Y)=G=\ol{G}$ is locally indicable.

\begin{thm}[Collins for complexes]\label{2ComplexVersion}
	$$ $$
	\vskip -0.8cm

	Suppose that $U_1,U_2,V_1,V_2, Q$ are   paths in $Y^{(1)}$ such that:
	\begin{enumerate}
		\item[] $U_1, U_2$ are closed paths in $X^{(1)}\cup\cS$ based at a vertex $v_0$;
		\item[] $V_1, V_2$ are closed paths in $X^{(1)}\cup\cT$ based at a vertex $v_1$;
		\item[] $Q$ is a path in $Y^{(1)}$  from $v_0$ to $v_1$;
		\item[] there are homotopies  rel endpoints between $Q^{-1}U_iQ$ and $V_i$ in $Y$ for $i=1,2$.
		\item[]   Let $K$ denote the subgroup of $\pi_1(Y,v_0)$ generated by $U_1$ and $U_2$.  
			Suppose in addition that either:
		\begin{enumerate}
			\item[(for Theorem \ref{main})] $ K$ is not cyclic; or
			\item[(for Theorem \ref{maintorsion})]  $m\ge 2$ and $K$ is non-trivial.
		\end{enumerate} 
	\end{enumerate}
	Then $Q$ is homotopic rel endpoints in $Y$ to a path $Q_1\cdot Q_2$ where $Q_1$ is in $X\cup\cS$ and
	$Q_2$ is in $X\cup\cT$. 		
\end{thm}

This is the theorem we shall prove.   
Theorems \ref{main} and \ref{maintorsion} follow  from Theorem \ref{2ComplexVersion} by taking:

$X_1$  a  one-vertex 2-complex with  fundamental group $A=\ast_{i\in I\setminus J} G_i$,  

$X_2$  a   connected two-vertex 2-complex with fundamental group $C=\ast_{i\in I\cap J} G_i$ 
 and an edge $x_0$ joining the two vertices,

$X_3$  a  one-vertex  2-complex with fundamental group $B=\ast_{j\in J\setminus I} G_j$, 

$X=X_1\cup X_2 \cup X_3$,

$\cS=\{s\}$ a single edge joining 
 one vertex of $X_2$ to 
 the vertex of $X_1$,

$\cT=\{t\}$ a single edge joining the  other vertex of $X_2$ to 
 the vertex of $X_3$,

$\alpha$ a 2-cell adjoined identifying its boundary with the path $R''$ obtained as follows:
write the word
	$R=a_1a_2\dots a_k$ in cyclically reduced form: where $a_i$ is in some factor $A$, $B$ or $C$, 
and $a_{i+1}$  is in a different factor from $a_i$ (subscripts modulo $k$)  
(in particular $a_i\neq 1$). 
Then the  	corresponding closed path is 
	 $R'= b_1\dots b_k$ where $b_i=s\g_is^{-1}$ if $a_i\in A$, $\g_i$ is a path in $X_1^{(1)}$ representing $a_i$, $b_i=t\g_it^{-1}$ if $a_i\in B$, $\g_i$ is a path in $X_3^{(1)}$ representing $a_i$, and
$b_i=\g_i$ if $a_i\in C$, $\g_i$ is a path in $X_2^{(1)}$ 
such that $a_i$ is represented by $g_i$ or $g_ix_0^{\pm1}$. 
 We may suppose that the endpoints of $\cS\cup\cT$ are pairwise distinct as required by the Reformulation.  (For example, trisect each edge $u\in\cS\cup\cT$ and replace it by its (closed) middle third. Then expand $X$ to include the closures of the other thirds of $\cS\cup\cT$.)  We may  also suppose that each $\g_i$ 
 has length $1$  to match the description in the Reformulation;
if necessary add an edge $x_i$ for the path $\g_i$,
and a 2-cell with boundary  $x_i \g_i^{-1}$. 

Theorems \ref{main} and \ref{maintorsion} now follow immediately from Theorem \ref{2ComplexVersion} 
and Lemma \ref{red1}
with $Y= X\cup\cS\cup\cT\cup\alpha$   where $\a$ is
attached along the path $R''$, the path obtained by cyclically reducing the path $R'$ in $Y^{(1)}$.
  
 Note that equality of group elements in our main theorems translates into homotopy (rel base-point or rel end-points) in the $2$-complex version.  We will often use the notation $\sim$ to denote these homotopy relations.
 
  We also record the following useful observation.
\begin{lem}
If Theorem \ref{main} fails, then there is a counterexample to Theorem \ref{2ComplexVersion} in which the closed paths $U_1,U_2$ generate a free subgroup of rank $2$ in $\pi_1(X\cup\cS,v_0)$ -- and therefore $V_1,V_2$ generate a free subgroup of rank $2$ in $\pi_1(X\cup\cT,v_1)$.
\end{lem} 	

\begin{proof}  
For each component  
 $X_\lambda$ of $X$,  
 $\pi_1(X_\lambda)$ is a free factor of $\pi_1(X\cup\cS\cup\cT)$, and
indeed  
$\pi_1(X\cup\cS\cup\cT) \cong \pi_1(X_\lambda) * D$ where $D$ is isomorphic to the free product of the
fundamental groups of the other components and a free group.
Since $R$ is not contained in  
 $X_\lambda$,
it follows that $D$ is locally indicable and not trivial,
and   
 $\pi_1(X_\lambda)$ is naturally included in  $\pi_1(Y)$ as a subgroup
by  Theorem \ref{Freiheitssatz}.

For any  component $N$ of $X\cup\cS$,  the fundamental group  (choosing a  base  point in $N$)
$\pi_1(N)$ is a free product of the fundamental groups of the components 
$X_\lambda$ in $N$ and  a free group,
and the Kurosh subgroup theorem says any subgroup 
is  a free product of conjugates in $\pi_1(N)$ of subgroups of the  
$\pi_1(X_\lambda)$ and a free group. 

If the subgroup $K$ generated by $U_1,U_2$ is  non-cyclic and contains no 
rank 2 free subgroup, then $K$ is conjugate in  $\pi_1(N,v_0)=\pi_1(X\cup\cS,v_0)$ 
(where $N$ is the component of $X\cup\cS$ with $v_0$ in $N$) 
to a  subgroup of $\pi_1(X_\lambda)$ for some $\lambda$. 
Similarly the subgroup $L$ generated by $V_1,V_2$ is conjugate  in   $\pi_1(N',v_1)=\pi_1(X\cup\cT,v_1)$ 
(where $N'$ is the component of 
$X\cup\cT$ with  $v_1$ in $N'$) 
to a subgroup of  some $\pi_1(X_\mu)$.
In this case   Theorem \ref{BrodskyLemma} applies 
to give that  $X_\lambda=X_\mu$ and the path $Q$ 
is homotopic rel endpoints into   
$X_\lambda$, which is stronger than the result claimed.
\end{proof}
 
\begin{rmk}\label{passtosubgroup}
The pair $U_1,U_2$ can in principle be replaced by any pair $U_1',U_2'$ that generates a rank $2$ subgroup of the free group $\<U_1,U_2\>$.  We will exploit this feature during the proof of Theorem \ref{2ComplexVersion}.
\end{rmk}

 \smallskip
\noindent{\bf Definition:} 
Let $Q$ be an edge path in $Y^{(1)}$: 
define a {\it $(Y,Q)$ Collins counter--example\/}, abbreviated to $(Y,Q)$-CCE,
to be  a quadruple\\ $(U_1, U_2,V_1,V_2)$ with $U_1, U_2$ closed edge paths in $X\cup\cS$ based at $v_0$, 
the initial vertex of $Q$, and
$V_1,V_2$ are closed edge paths in $X\cup\cT$ based at $v_1$, the final vertex of $Q$.
Moreover   for $i=1,2$, the closed path $QV_iQ^{-1}U_i^{-1}$ is nullhomotopic in $Y$, 
and $U_1,U_2$   generate a rank 2 free subgroup of 
$\pi_1(Y,v_0)$. 
The $(Y,Q)$-CCE is {\it trivial\/} if $Q$ is   homotopic rel end-points to a concatenation $Q_1\cdot Q_2$
with $Q_1$ in $X\cup\cS$ and $Q_2$ in $X\cup\cT$.

\smallskip
The task now  for the proof of Theorem \ref{2ComplexVersion} in the torsion-free case 
-- and hence also of Theorem \ref{main} -- is to prove that there are no
non-trivial CCEs.

\subsection{Pictures} 

We shall  make use of the duals of van Kampen diagrams, 
known as {\em pictures}.   
Originally a van Kampen diagram over a presentation was 
a finite decomposition of a surface (usually a disc or an annulus)
as a 2-complex mapping cellularly to the standard 2-complex for the presentation.
Pictures were introduced by Rourke in \cite{Rou} and studied in  the relative context by the second-named author \cite{S81}. 
We generalise slightly the original definition.
Let  $f:\Sigma\to Y$ be a continuous map from  a compact orientable surface  $\Sigma$,  
	here a disc representing an identity in $\pi_1(Y)$ 
	or an annulus representing a conjugacy relation in $\pi_1(Y)$, to  a $2$-complex $Y$, made transverse to the centres of the 2--cells of $Y$ and
	to the mid--points of the 1--cells of $Y$.
	The preimages of neighbourhoods of the centres of the 2--cells form a disjoint collection of 
	{\em discs} or {\em (fat) vertices} in $\Sigma$, and after a minor adjustment, 
	the preimages of the mid--points of the 1--cells form a properly embedded $1$-submanifold of the complement of the interiors of the discs, 
	each component of which is called an {\em arc}, carries a transverse orientation and is labelled by a $1$-cell of $X$.
A small regular neighbourhood of each arc is mapped to the corresponding $1$-cell in the direction of the transverse orientation.  
Reading the labels around a fat vertex gives the boundary of the corresponding $2$-cell of $Y$, up to cyclic permutation and inversion.    Reading labels around the boundary of a disc-picture (resp. the two boundary components of an annular picture) gives a nullhomotopic closed path in $Y$ (resp. a freely homotopic pair of closed paths in $Y$). 

Alternatively a picture can be obtained from 
 its dual van Kampen diagram by placing a disc in the interior of each face of the diagram and an arc transverse to each edge in the boundary of a face joining the midpoint of the edge to the disc.

A picture can be simplified 
if there is an arc joining two discs, 
such that reading in opposite directions round the two disks from the endpoints of the arc gives the same word.
The two disks and the arc joining them can then be removed and the remaining arcs joined up in a coherent way.
We also call this operation {\em cancellation} of the two disks.
A picture is said to be {\em reduced} if this operation cannot be performed.

The labelled graph on $\Sigma$ formed of the fat vertices and the labelled transversely oriented arcs is the picture over $Y$ of the map $f$.   
A picture is said to be {\em connected} if this graph is connected.   (Note that this property is distinct from that of connectedness of the ambient surface $\Sigma$.)
For more details on pictures, and examples of their usefulness in group theory we refer the reader to \cite{DH}, \cite{H4}.
The  generalisation of the original pictures is that we do not insist that the 1--cells in $Y$ be
loops.

For the proofs of our main theorems, 
we are interested here in  pictures over 
$Y=X\cup\cS\cup\cT\cup\alpha$, where $X$ is a 2--complex 
(not necessarily connected),  $\cS,\cT$ are sets of edges, and $\alpha$ is a 2-cell, such that $Y$ is connected.
We get a {\it relative picture\/} over $Y$  (relative to $\cS\cup\cT\cup\alpha$)
by removing from the picture all discs that do not
map to $\alpha$, and all arcs other than those labelled in $\cS\cup\cT$.
In particular we shall be interested in  homotopies 
in $Y$ 
of the form $U\sim Q_1VQ_2^{-1}$ where $U$ is a loop in $X\cup \cS$ based at $v_0$,
$V$ is a loop in $X\cup\cT$ based at $v_1$,
and $Q_1,Q_2$  are paths in $Y^{(1)}$ from $v_0$ to $v_1$. 
There is an obvious way to draw a corresponding picture 
as a rectangle  with top side labelled $U$, bottom side labelled $V$ and vertical sides labelled $Q_1,Q_2$.

\begin{figure}[h]
	\hskip -0.1cm\includegraphics[scale=0.22] {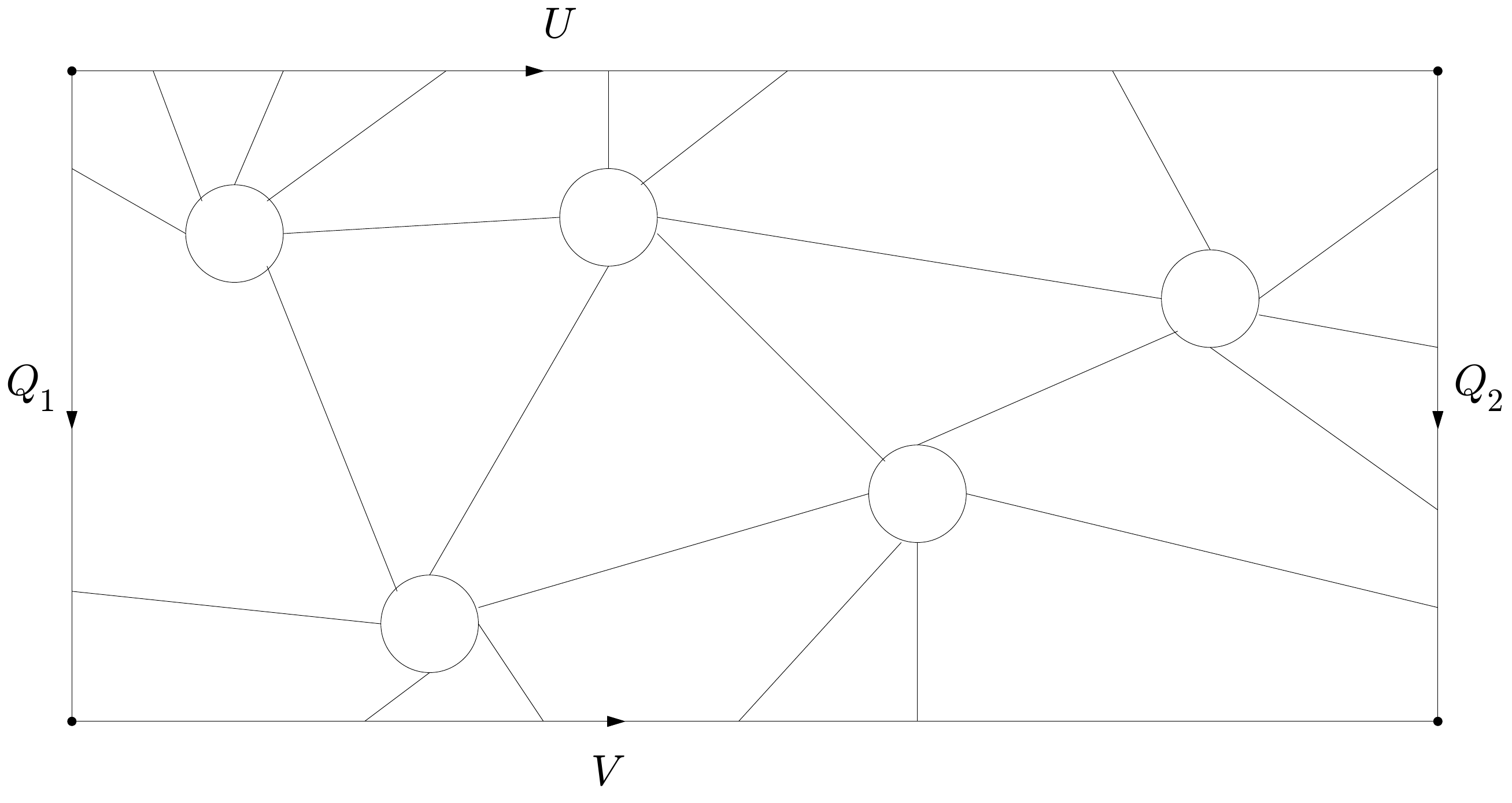}
	\caption{The top side is labelled by $U\in X\cup\cS$, the bottom by $V\in X\cup\cT$ }
\end{figure}

When $Q_2=Q_1$, the vertical sides can be identified to give a   picture on 
an annulus.
We refer to these as {\em rectangular} and {\em annular} relative  pictures. 
Given two such rectangular pictures   $\G,\G'$ for  $U\sim Q_1VQ_2^{-1}$ and for $U'\sim Q_2V'Q_3^{-1}$,
the right hand side of one can be identified with the left hand side of the other to
give a rectangular picture for the  
nullhomotopy $UU'\sim Q_1(VV')Q_3^{-1}$ 
 which (after reduction of cancelling   disks) 
 we denote $\G+\G'$. 
Indeed the set of such rectangular pictures  
(modulo a suitable homotopy relation) can be given a groupoid structure with
this gluing operation. 
Reflecting the rectangular picture $\Gamma$ in a vertical mirror provides 
an inverse, which we write $-\Gamma$.

	\subsection{Syllables and syllable length} 
	In analogy with a commonly used terminology for words in a free product with amalgamation, we will use the term {\em syllable} to denote a subpath of a path in $X\cup\cS\cup\cT$ which contains $\cS$-edges or $\cT$-edges but not both, and is maximal with respect to that property -- except possibly for an initial and/or a terminal sub-path that is  contained in $X$.  
	We will speak of $\cS$-syllables and $\cT$-syllables with the obvious meaning.   The {\em syllable length} $\text{SL}(P)$ of a path $P$ is defined to be the number of syllables into which it can be subdivided.   

\smallskip  
We will also use another measure of length: the {\em $\cS\cup\cT$ length} of a path in $Y^{(1)}$ is the number of occurrences of edges from $\cS^{\pm1}\cup\cT^{\pm1}$.

\subsection{Strong reduction}

An edge-path in $X^{(1)}\cup\cS\cup\cT$ is said to be {\em strongly reduced} if it is reduced and has no subpath of form $u^\varepsilon\cdot\g\cdot u^{-\varepsilon}$ with $u\in\cS\cup\cT$, $\varepsilon=\pm1$, and $\g$ a closed path in $X^{(1)}$ that is nullhomotopic in $X$. Every path in $X^{(1)}\cup\cS\cup\cT$ is homotopic rel end-points in $X\cup\cS\cup\cT$ to a strongly reduced path.  A closed path in $X^{(1)}\cup\cS\cup\cT$ is {\em strongly cyclically reduced} if each of its cyclic subpaths is strongly reduced.  Every closed path in $X^{(1)}\cup\cS\cup\cT$ is freely homotopic in $X\cup\cS\cup\cT$ to a strongly cyclically reduced path.
For example,  the closed path $\partial\a$ in $X^{(1)}\cup\cS\cup\cT$ is strongly cyclically reduced.

\medskip
	
	We introduce a move on rectangular and annular relative pictures called {\em boundary surgery}. 
	If there are two adjacent ends of arcs on the top or bottom boundary with  the same label, opposite orientations, and  separated on the boundary by a path labelled by a nullhomotopic closed path in $X$, 
	then remove small neighbourhoods of the endpoints of the arcs, and extend the remaining parts of the arcs to join in the picture.
	The inverse operation can be realised when there is a path in pictures from an interior point of an arc to the top or bottom boundary meeting no other arc (or vertex). The original arc is extended along both sides of the path to introduce two new
	endpoints of arcs on the boundary.   
    Throughout the paper, the term boundary surgery will refer to any combination of moves of these two forms. 

We shall suppose that our annular pictures have 
	strongly reduced boundaries. 
For an annular  picture it may be that the top or bottom label is not strongly cyclically reduced, 
even if it was formed from a 
rectangular picture with strongly reduced labels on each side.  
We may again perform boundary surgeries on such a picture.   
If the original picture represented a conjugacy equation   $U\sim QVQ^{-1}$
with $U$ based at $v_0$ and
$V$ based at $v_1$   
then the surgered version will represent a slightly different conjugacy equation 
 $U'\sim Q'V'{Q'}^{-1}$.  
Here  $U'$ 
has the form $U'=\sigma U\sigma^{-1}$  with 
$U'$ a loop based at a vertex $v_0'$ and  $\sigma$ a path in $X^{(1)}\cup\cS$  
from $v_0'$   to $v_0$.     
Similarly $V'=\tau^{-1}V\tau$ with $V'$ a loop in $X^{(1)}\cup\cT$ based at a vertex $v_1'$
and   $\tau$ a path in $X^{(1)}\cup\cT$ 
from  $v_1$   to $v_1'$.
And $Q'$ is the path $\sigma Q\tau$ in $Y^{(1)}$ from $v'_0$ to $v'_1$.

Abusing notation, we shall write $Q\pi_1(X\cup\cT,v_1)Q^{-1}$ for the subgroup of $\pi_1(Y,v_0)$
of the homotopy classes $[Q\gamma Q^{-1}]$ for all $[\gamma]\in\pi_1(X\cup\cT,v_1)$.

\subsection{Brodski\u\i's Theorem(s) for $2$-complexes}

We reformulate  Theorems \ref{BrodskyLemma}   and \ref{BrodskyTorsion}  for 2--complexes as follows:

\begin{lem}[Brodski\u\i's Theorem for complexes]\label{Br}
	Let $\mathcal{Y}$ be  a  $2$-complex of the form $\mathcal{Y}=X\cup E\cup \a$	 
	for some non-empty set $E$ of $1$-cells and $\a$ a $2$-cell, 
	where each component of $X$ has locally indicable  (possibly trivial) fundamental group.  
	Assume that $\partial\a$ is not freely homotopic in $X\cup E$ 
	to a path in $X$.   
	Suppose that $Q$ is an edge-path in $\mathcal{Y}$ from a $0$-cell $*_0$ to a $0$-cell $*_1$, 
	and that $H_0<\pi_1(X,*_0)$, $ H_1< \pi_1(X,*_1)$ are 
	subgroups such that $H_0=Q  H_1Q^{-1}$ in $\pi_1(\mathcal{Y},*_0)$,  
	\begin{itemize}
		\item $ H_0$ is non-cyclic; or
		\item $\partial\a$ is freely homotopic in $X\cup E$ to a proper power, and $H_0$ is non-trivial.
	\end{itemize} 
	Then $Q$ is homotopic rel endpoints in $\mathcal{Y}$ to an edge-path  $Q_0$ in $X$  
	such that $H_0=Q_0 H_1 Q_0^{-1}$ in $\pi_1(X,*_0)$.
\end{lem}

Note: the conditions imply that at least one component  of $X$ has a fundamental group 
that is non-trivial
(indeed  in the torsion-free case it contains a non-cyclic subgroup) and the conclusion implies that $*_0$ and $*_1$ are contained in the same
component of $X$.

\begin{proof}
	We  work throughout within the component of $\mathcal{Y}$ that contains $Q$, so without loss of generality we may assume that $\mathcal{Y}$ is connected. 
	
	Let  $X_{\lambda}$ ($\lambda\in\Lambda$) be the components of $X$. 
	For each $\lambda\in\Lambda$ choose a base-vertex $v_\lambda\in X_{\lambda}$ and a spanning tree $T_\lambda$ in
	 $X_{\lambda}$.   
	 The forest $\cup_\lambda T_\lambda$ contains all the $0$-cells of $\mathcal{Y}$, so can be extended to a spanning tree $T$ for $\mathcal{Y}$ by adding a subset of  $E$.
	
	Let $P_\lambda$ ($\lambda\in\Lambda$) and $P_\a$ denote respectively the unique geodesics in $T$ from $*_0$ to $v_\lambda$ and from $*_0$ to the initial point in $\mathcal Y$ of the closed path $\partial\a$. Then $\pi_1(\mathcal{Y},*_0)$ is the one-relator product 
	$$\frac{\left(\ast_\lambda G_\lambda\right)*F}{\<\<R\>\>}$$
	of the locally indicable groups $G_\lambda:=P_\lambda\cdot\pi_1(X,v_\lambda)\cdot P_\lambda^{-1}$ and the free group $F$ on $E\setminus E(T)$, where $R$ is the group element represented by $P_\a\cdot\partial\a\cdot P_\a^{-1}$.
	
	Let $\lambda(0),\lambda(1)\in\Lambda$ denote the indices such that $*_j\in X_{\lambda(j)}$ for $j=0,1$.
	Since $P_{\lambda(0)}$ is contained in $X_{\lambda(0)}$ it follows that $G_{\lambda(0)}=\pi_1(X,*_0)$.
	Similarly, if $P_*$ is the geodesic in $T_{\lambda(1)}$ from $*_1$ to $v_{\lambda(1)}$, then 
	$$\pi_1(X,*_1)=P_*\cdot\pi_1(X,v_{\lambda(1)})\cdot P_*^{-1}=P_*\cdot P_{\lambda(1)}^{-1}\cdot G_{\lambda(1)}\cdot P_{\lambda(1)}\cdot P_*^{-1}.$$
	
	Since $H_0<\pi_1(X,*_0)$ and $H_1<\pi_1(X,*_1)$ are non-cyclic   -- or, in the torsion case, non-trivial --
	it follows that the intersection of $G_{\lambda(0)}$ and $gG_{\lambda(1)}g^{-1}$ 
	in   $\pi_1(\mathcal{Y},v_0)$ is non-cyclic  (resp. non-trivial), 
	where $g\in\pi_1(\mathcal{Y},*_0)$ is the group element represented by the closed edge-path $Q\cdot P_*\cdot P_{\lambda(1)}^{-1}$.
	By Brodski\u\i's original version from \cite{B81}, Theorem \ref{BrodskyLemma} 
	(resp. the torsion version from \cite{HS}, Theorem \ref{BrodskyTorsion}),
	it follows that $\lambda(1)=\lambda(0)$ and $g\in G_{\lambda(0)}$.  
	It follows that $P_*$ and $P_{\lambda(1)}$ are paths in $T_{\lambda(0)}\subset X$; 
	and that $Q\cdot P_*\cdot P_{\lambda(1)}^{-1}$ is 
	homotopic rel basepoint in $\mathcal{Y}$ to a path $P'$ in $X$ representing $g$.  Hence $Q$ is homotopic rel endpoints in $\mathcal{Y}$ to the path $Q_0:=P'\cdot P_{\lambda(1)}\cdot P_*^{-1}$ in $X$.   
	Finally, by the Freiheitssatz,  Theorem \ref{Freiheitssatz}, 
	$H_0=Q_0 H_1Q_0^{-1}$ in $\pi_1(X,*_0)$ as required.
\end{proof}

\subsection{Stallings graphs} 
For the proof of Theorem \ref{main} 
we are going to examine rank $2$ 
free subgroups $K<\pi_1(X\cup\cS,v_0)$ and $L<\pi_1(X\cup\cT,v_1)$ that are conjugate in $G$.

Recall from the Reformulation in \S \ref{refandreds} 
 that the boundary of the added 2--cell $\alpha$ is attached along  a path 
$R=Z^m$ and $Z=u_1x_1u_2x_2\dots u_kx_k$ where each $u_i$ is $e^{\pm1}$ for some edge $e\in\cS\cup\cT$, 
and each $x_i$ is an edge-path   of length  $1$ in (some component of) $X$.

Then any basis for $K$ can be expressed as a pair of based, reduced closed paths   $\{U_1,U_2\}$ in ${X}^{(1)}\cup\cS$. This gives two maps of circles into $X\cup\cS$, which can be decomposed into
1-complexes with edges labelled by  the images in $X^{(1)}\cup \cS$, and base-vertices mapping to $v_0$.
Identify base-vertices as $*$ and fold edges with the same image to make the map an immersion.
The resulting labelled graph we shall call the {\em Stallings graph} $St(U_1,U_2)$ (or $St(K)$), and the resulting immersion is denoted $\iota_K:(St(U_1,U_2),*) \to (X\cup\cS,v_0) \subset (Y,v_0)$.
We shall suppose that the Stallings graphs considered here are more than strongly reduced: 
if the image of a   non--empty reduced path  in $St(U_1,U_2)$   is a closed path in $X^{(1)}$ 
then this loop is not nullhomotopic in $X$.  
(If  a path contradicting this property exists, we can identify its endpoints and then remove one of its edges, 
to obtain a smaller graph with different $U_1, U_2$, but generating the same subgroup.) 
Also remove free edges that do not contain the base point.
As there only finitely many edges in the graph, after finitely many such moves the required property will be satisfied.
We will later make implicit use of this feature by assuming that the top and bottom labels of the pictures that we study are strongly reduced.

Similarly there is an immersion $\iota_L:(St(V_1,V_2),*') \to (X\cup\cT,v_1)\subset (Y,v_1)$.

This is  a slight generalisation of a construction of Stallings \cite{St},
where he showed 
that every finitely-generated subgroup $H$ of a free group $F$
on a given basis $\mathcal{Y}$
can be realised by an immersion of a graph to the rose with petals labelled by the elements of 
$\mathcal{Y}$.

Recall that a finite 
connected graph is a {\em core graph} if it has no vertices of degree $1$. 
Every finite connected graph $\G$ has a unique core subgraph. 
Stallings graphs are not in general core graphs.  
However, if $St(U_1,U_2)$ is not a core graph, then the only vertex of degree $1$ in $St(U_1,U_2)$ 
is the base-point $*$;  
let $\g$ in $St(K)$ be the path from its base point to a vertex in the core. 
   Then $Q_0:=\iota_K(\g)$ is  a path in $X^{(1)}\cup\cS$ from $v_0$ to a vertex $v'_0$.  Replace $St(K)$ by its core and $Q$ by $Q_0^{-1}Q$ in the theorem.  
   Similarly $V_1,V_2$ give a Stallings graph $St(L)$ which we may assume is core, 
   and an immersion $\iota_L:St(L)\to Y$ based at $v_1$.
   Thus up to replacing $Q$ by a path  $Q_0^{-1}QQ_1$ with $Q_0$ in $X\cup \cS$ and $Q_1$ in $X\cup\cT$,
   we can assume that the Stallings graphs are core.

\subsection{Orderings}\label{ords}

Recall from  \S \ref{refandreds} that $Y=X\cup\cS\cup\cT\cup\alpha$, where
the closed path  $R$  identified with $\partial\alpha$ can be written  $R=Z^m=(u_1x_1u_2x_2\dots u_{k}x_k)^m$,
and $Z$ is a cyclically reduced loop at $*$,   not a proper power in   $\pi_1(X\cup\cS\cup\cT)$  and not contained in $X$.
Then $\ol{G}:=\pi_1(X\cup\cS\cup\cT\cup\alpha',*)$ where $\alpha'$ is a 2-cell attached along the loop $Z$,
is locally indicable (Theorem \ref{Gli}), and hence right orderable (Theorem \ref{lilo}).   
 Choose a right ordering $<$ on $\ol{G}$.

\smallskip
For each  $u\in(\cS\cup\cT)^{\pm1}$, choose a path $\tau(u)$ in $Y^{(1)}$ 
from the initial point of $u$ to the base point $*$ of $Y$.
(For instance choose a maximal tree in $Y^{(1)}$ and a path therein.)
For each $j=1,\dots k$, define $z(j) = u_1x_1\dots x_{j-1}$ and $\tau_j=\tau(u_j)$ if
$u_j\in\cS\cup\cT$, and $z(j) = u_1x_1\dots x_ju_j$ and   $\tau_j=\tau(u_j^{-1})$ if $u_j\in\cS^{-1}\cup\cT^{-1}$.
The path $z(j)\tau_j$ is  a closed path in $Y$ based at $*$ : let $g_j\in\pi_1(Y,*)$
be the corresponding group element.
For each $u\in \cS\cup\cT$ consider the occurrences of $u^{\pm1}$ in $Z$ :  
let $\text{Ind}(Z,u)=\{j \mid  u_j =u^{\pm1}\}$.
Thus if $i,j\in \text{Ind}(Z,u)$  
then $\tau_i=\tau_j$ and $ z(j)\tau_j(z(i)\tau_i)^{-1}$ reduces to $z(j)z(i)^{-1}$,
a loop in $Y$ based at  $*$, and the corresponding group element is $g_jg_i^{-1}$.
The occurrence $u_{max}$ is at $j\in \text{Ind}(Z,u)$ when $g_j\ge g_i$ for all $i\in \text{Ind}(Z,u)$.
The occurrence is unique by Weinbaum's Theorem \ref{Weinbaum} as $g_j=g_{j'}$ with $j,j'\in \text{Ind}(Z,u)$
implies that a subword of $Z$ corresponding to  $g_j^{-1}g_{j'}$ is trivial.
Similarly there is a unique occurrence $u_{min}$ of each $u\in\cS\cup\cT$, with the obvious meaning.   
Note that each of $u_{max},u_{min}$ is repeated precisely $m$ times in $R=Z^m$.

Observe that in a picture over $Y$, if the $u_{max}$ arc of one $\alpha$ disc arrives as a $u_{max}$ arc on 
another $\alpha$ disc, then the two discs cancel.   
An analogous remark holds for the $u_{min}$ arc.

The right ordering on $\ol G$  induces a pre-ordering on $\a$-cells 
in any 
relative  picture $P$ over $Y$ 
 on a disc:  
$\b_1\le\b_2$ if the label on some (and hence any) path in the picture  
from the base-point of $\b_1$ to that of $\b_2$,  is $\ge 1$ in $\ol{G}$.
Note that if a $u_{max}$  arc of the $\a$-cell $\beta_1$ has its other endpoint on the
$\a$-cell $\beta_2$ as $u_i$, then  $\beta_1 \le \beta_2$:
there is a path $z(j)z(i)^{-1}$ with $u_j$ the occurrence of $u_{max}$, 
from the base point of $\beta_1$ to the base point of $\beta_2$, and 
$g_{j} g_i^{-1} \ge 1 \iff g_j \ge  g_i$, which is the case at $u_{max}$.

Alternatively, lift $P$ to the regular cover of $Y$ with deck transformation group $\ol{G}$. If $\b_1,\b_2$ lift to $\wh\b_1,\wh\b_2$ respectively, then $\b_2=\gamma(\b_1)$ for a unique $\gamma\in\ol G$.  We say that $\b_1<\b_2$ if $\gamma>1$ in $\ol G$.  

When we speak of minimal or maximal cells in a 
disc-picture, we mean with respect to this pre-order.

\smallskip
	The following lemma will be used in the proofs  of   our main results.

\begin{lem}\label{l:1}
	Let $\G$ be a reduced relative picture   on a disc $D$
	with more than one   $\a$-disc 
	over $Y = X\cup \cS\cup\cT\cup\alpha$. 
	Then there are two  $\a$-discs 
	$\b_1,\b_2$ in $\G$, each of which   is joined to the boundary of $D$ (only)  
	by a sequence of consecutive arcs 
	which contains 
	either all of its $u_{min}$   arcs   for every $u\in\cS\cup\cT$  
	or all of its $u_{max}$  arcs for every $u\in\cS\cup\cT$.  
\end{lem}

\begin{proof}
	Let $\b,\b'$ be a minimal and a maximal $\a$-disc 
	respectively in $\G$.  
	If $\b,\b'$ satisfy the conclusion in the statement then we are done.  
	Otherwise at least one of them -- say $\b$ -- fails.   
	Since $\b$ is minimal, 
 none of its $s_{min}$ and $t_{min}$ arcs meet other $\a$-discs of $\G$,
	for all $s\in\cS, t\in\cT$.   
	And since $\G$ is reduced all of these arcs of $\b$ end on $\partial D$.
	Thus the sequence of arcs from $\b$ to the boundary is not consecutive, 
	and so  
		$\b$, together with its $s_{min},t_{min}$ arcs, divides $\G$ into at least two smaller  pictures.

	We have $\G=\G_1\cup\G_2$ with $\b\subset\G_1\cap\G_2$, where each of $\G_i$ has at least two $\a$-discs, but   fewer $\a$-discs than $\G$.  
	Inductively, we may assume that the result applies to each $\G_i$: 
	so $\G_i$ has an $\a$-disc  
	$\b_i\ne\b$ satisfying the conclusion of the Lemma.  
	Then the pair $\{\b_1,\b_2\}$ will do.
\end{proof}

\section{The Torsion Case: Proof of Theorem \ref{maintorsion}}\label{torsion}

	Let $Y=X\cup\cS\cup\cT\cup \alpha$ be as in the statement of 
 Theorem \ref{2ComplexVersion},
where the boundary of $\alpha$ is identified with the loop $R=Z^m$ in $Y^{(1)}$ and $m\ge 2$.
We show that if there are loops $U$ based at $v_0$ in $X\cup\cS$ and 
$V$ based at $v_1$ in $X\cup\cT$,
not nullhomotopic in $Y$, 
and a path $Q$ such that $U$ is homotopic to $QVQ^{-1}$ in $Y$,
then $Q$ is homotopic rel endpoints in $Y$ to a path $Q_1Q_2$ with $Q_1$ in $X\cup\cS$ and
$Q_2$ in $X\cup\cT$.

\smallskip
Let $\G^{rect}$ be a reduced rectangular relative picture representing 
a homotopy between $U$ and  $QVQ^{-1}$.  
Thus $\G^{rect}$ has both side labels $Q$, top label $U$ and bottom label $V$.   
Form an annular picture $\G^{ann}$ by identifying the two vertical sides of $\G^{rect}$ 
and then performing $\a$-cell cancellations and boundary surgery to ensure that $\G^{ann}$ is reduced 
and its top and bottom boundary labels are strongly cyclically reduced.

Now choose a path $\g$ between the two boundary components of the annulus that meets $\G^{ann}$ transversely and minimally.
In other words, $\g$ meets none of the $\a$-discs of $\G^{ann}$, and intersects the union of the $(\cS\cup\cT)$-arcs transversely in the fewest possible points.  
If $Q'$ is the path in $Y^{(1)}$ labelling $\g$, then there are paths $P_\cS,P_\cT$ in $X\cup\cS$ 
and $X\cup\cT$ respectively such that $Q'$ is homotopic rel endpoints to $P_\cS\cdot Q\cdot P_\cT$.  
 Hence we may assume that $Q=Q'$ and that $\G^{rect}$ can be recovered from $\G^{ann}$ by cutting the annulus along $\g$.

Note first that we may assume that there are arcs in $\G^{ann}$ 
going to each of the two boundary components of the annulus: if for example there are no arcs meeting the top boundary, 
then $Q$ is homotopic to a path in $X\cup\cT$ by Lemma \ref{Br}, since $U$ is not nullhomotopic in $X\cup\cS$.  
 Recall that the closed path $Z$ traverses an edge of $\cS\cup\cT$ precisely $k$ times, 
 so that each $\a$-disc in a picture is incident to  precisely $mk$ $(\cS\cup\cT)$-arcs.   
 For any integer $N>1$, the rectangular picture
$$N.\G^{rect}:=\G^{rect}+\cdots+\G^{rect}~~(N~\mathrm{terms}),$$
representing 
a homotopy between $U^N$ and  $QV^NQ^{-1}$, is reduced and has at least $N$ arcs going to each of the top and bottom boundaries.   Replacing $U$, $V$ and $\G^{rect}$ by $U^N$, $V^N$ and $N.\G^{rect}$ for sufficiently large $N$ (for example, $N>2mk$), it follows that $\G^{rect}$ has more than one $\a$-disc, and that each boundary component of the annulus meets arcs that go to neither of the $\a$-discs $\b_1,\b_2$ given by Lemma \ref{l:1}.   We use this fact to obtain a contradiction as follows.

For a chosen right ordering on $\ol{G}:=\pi_1(X\cup\cS\cup\cT)/\<\<Z\>\>$, 
there is a sequence of consecutive arcs from $\b_1$ to the boundary that contains either all its $u_{min}$ arcs for all $u\in\cS\cup\cT$, or all its $u_{max}$  arcs for all $u\in\cS\cup\cT$.  Assume the latter.   This sequence of arcs at $\b_1$ cannot contain all the top boundary arcs of $\G^{rect}$, nor all its bottom boundary arcs.  Hence it cannot contain arcs going to both the left and right sides.  Without loss of generality, assume that none of these arcs go to the right side.  

Without loss of generality we may also assume that $Z$ begins with a $\cS$-syllable and ends with a $\cT$-syllable, and 
thus the sequence of arcs spans a cyclic subpath of $\partial\a=Z^m$ of syllable length at least $(m-1)\text{SL}(Z)+2$
(reading from the first occurrence of $s_{max}$ with $s\in\cS$ the sequence contains $(m-1)\text{SL}(Z)+1$ syllables, and there is at least one more  $\cT$ 
syllable in the sequence). 
Moreover any subsequence going to the top (resp. bottom) boundary spans a subpath of a single $\cS$-syllable (resp. $\cT$-syllable) of $Z$.   
Thus the subsequence of consecutive arcs going from $\b_1$ to the left side of $\G^{rect}$ 
spans   at least 
$(m-1)\text{SL}(Z)$ complete  syllables of $R$, and hence contains at least $(m-1)k$ arcs.   
Now let $\g'$ denote the left side of the sub-picture obtained by removing $\b_1$ and its incident boundary arcs from $\G^{rect}$.   Then $\g'$ is a path between the two boundary components of $\G^{ann}$.  We claim that it has smaller transverse intersection with $\G^{ann}$ than $\g$, giving the required contradiction.

To see this, note that in replacing $\g$ by $\g'$, we replace the  transverse intersections of $\g$ with arcs going to the left side of $\G^{rect}$ (of which there are at least $(m-1)k$) by transverse intersections of $\g'$ with arcs from $\b_1$ that do not belong to the consecutive sequence going to the boundary of $\G^{rect}$ (of which there are at most $k-|\cS\cup\cT|$).  This proves the claim, and hence the result.
\hfill $\square$

\section{Plan of the proof of Theorem \ref{main}, and first steps} \label{plan} 

Recall from the Reformulation in \S \ref{prelims} that we have   a 2-complex of the form
$Y := X \cup\cS \cup\cT \cup\a$ where each component of $X$ has locally indicable fundamental group, 
$\cS$ and $\cT$ are two disjoint, non-empty sets of $1$-cells, 
and $\a$ is a $2$-cell attached along a strongly cyclically reduced closed path $R$ containing all the $1$-cells in $\cS\cup\cT$.  From now on we assume that we are in the torsion-free case, in which $R$ does not represent a proper power in $\pi_1(X\cup\cS\cup\cT)$ (so $G=\ol G$).
 
In this section we begin the proof of Theorem \ref{main}, in the form of the torsion-free case of Theorem \ref{2ComplexVersion}.   We first prove some preliminary results, then introduce an adjustment to the general form of a putative counterexample. This adjustment enables us to formulate an inductive process for the proof of the theorem.   Finally in this section we prove the initial case of the induction,  and explain how the argument for the inductive step splits into two separate cases. The first of these  will be treated in \S \ref{first} and the second in \S\S \ref{restrict}-\ref{finalcurtain}.

We define the  {\em complexity} of $Y$ to be $c(Y ) := k - \vert\pi_0(X)\vert$, where
$\pi_0(X)$ means the set of components of $X$.

We say that $Y$ satisfies the Collins property if 
there exists no non-trivial $(Y,Q)$-CCE.
Theorem \ref{main} (that is, the torsion-free part of Theorem \ref{2ComplexVersion}) can be reformulated then as:  
\begin{thm}  All 2--complexes  $Y$  as above satisfy the Collins property.
\end{thm}

The proof is by induction on $c(Y)$, and breaks down into several cases. 
Suppose that there is a $(Y,Q)$-CCE with $Y$ of smallest complexity, 
and for this $Y$, $Q$ is a shortest (w.r.t. $\cS\cup\cT$ length) path giving a $(Y,Q)$-CCE. 
The aim is to factor
nullhomotopies through a space of smaller complexity in order to apply
the inductive hypothesis. 
This space will be contained in a $\Z$-cover,
corresponding to some epimorphism onto $\Z$. 
In order to ensure the
existence of a suitable epimorphism – and hence $\Z$-cover, we  make
 adjustments to the complex $Y$. 
In order to justify these, we need the following result.
 
	\begin{lem}\label{Agenerates} 
		Let $Q$ be a shortest (w.r.t. $\cS\cup\cT$ length) path in $Y$ 
		so that there is a non-trivial $(Y,Q)$-CCE.
		Then there is a non-trivial $(Y,Q')$-CCE  $(U_1',U_2',V_1',V_2')$ 
		with connected reduced rectangular relative pictures $P_i'$ for $U_i'\sim Q'V_i'{Q'}^{-1}$, $i=1,2$.
	\end{lem}
	
	\begin{proof}
		Let $U_1,U_2,V_1,V_2$  be a $(Y,Q)$-CCE, with reduced relative pictures $P_i$ for
		 $U_i\sim QV_iQ^{-1}$, $i=1,2$.
	 
		First note that we can assume that $Q$ begins with a $\cT$ edge and ends with an $\cS$ edge,
		and that the $U_i$ (resp. $V_i$) are reduced and have the smallest number of $\cS$ 
		(resp. $\cT$) edges in their homotopy classes. 		Write $Q=t^{\epsilon_1}a_1u_1a_2u_2\dots a_\ell s^{\epsilon_2}$ where
		$u_i\in (\cS\cup\cT)^{\pm 1}$, $\epsilon_i=\pm 1$ and
		each $a_i$ is a path (possibly of length $0$) in some component $X_{i'}$ of $X$.
		
A {\em region} $\sigma$ of a (relative) picture $P$ is a component of the complement of the fat vertices and arcs of the picture in its ambient surface.  	In our rectangular pictures $P_i$ , 
	the arcs are preimages of the midpoints of the 
	edges in $\cS\cup\cT$ under the nullhomotopy map from the rectangle into $Y$, so each region $\sigma$ maps into  
	$X_\lambda^+$,
	some component $X_{\lambda}$ of $X$ together with 
	half edges (from those edges in $\cS\cup\cT$ meeting $X_{\lambda}$) adjoined.

If $\delta$ is a simple closed curve in a region $\sigma$, 
then $\delta$ bounds a disk $\Delta$ in the rectangle.  
So $\delta$ is mapped onto a closed curve in $X^+:=X\cup((\cS\cup\cT) \setminus {\rm midpoints})$ that is nullhomotopic in $Y$ -- and hence already nullhomotopic in $X^+$, 
indeed in $X_\lambda^+$ for some $\lambda$, by the Freiheitssatz (Theorem \ref{Freiheitssatz}).  So we may assume that none of the vertices or arcs of the picture are contained in the disk $\Delta$ -- in other words that $\sigma$ is simply connected.  Moreover, taking $\delta=\partial\sigma$ in the above remark, we see that the label on $\partial\sigma$ is nullhomotopic in some $X_\lambda^+$. 

	Suppose that $P_i$ is not connected.	Since every region is simply connected,  there is a region $\sigma$ which meets the boundary of the rectangle 
in more than one connected segment,
		separating $P_i$ into two non--empty sub-pictures.
		
			If $\sigma$ meets one of the four sides of $P_i$ in two disjoint segments, 
		then we can shorten $U_i$  or $V_i$ or $Q$, replacing some sub-path by a path in $X$.  This contradicts the choice of words.

	Next suppose  that $\sigma$ meets the left side of the rectangle in the segment labelled $a_j$, and the top side following immediately after the initial subpath $U_0$ of $U_i$.  Then the top-left corner cut off by $\sigma$ gives a nullhomotopy of $t^{\epsilon_1}a_1u_1\cdots u_{j-1}\gamma U_0^{-1}$ for some path $\gamma$ in $X$ (the label on a path in $\sigma$). 
	(See Figure \ref{sigma1}.)
	 But then 
$$(\gamma U_0^{-1}U_1U_0\gamma^{-1},\gamma U_0^{-1}U_2U_0\gamma^{-1},V_1,V_2)$$
is a $(Y,Q_0)$-CCE where $Q_0:=a_ju_j\cdots a_\ell s^{\epsilon_2}$ is shorter than $Q$ -- again a contradiction.  Similar observations apply when a region $\sigma$ meets either vertical side and either horizontal side, or when it meets both the top and the bottom.
		It follows that $\sigma$ meets $\partial P_i$ in exactly two connected segments, 
		one on the right hand side labelled  $a_j$, the other on the left hand side labelled $a_{j'}$. 
		Call this  an $A$--region at $(j,j')$.

\begin{figure}[h]
	\hskip -0.6cm\includegraphics[scale=0.25] {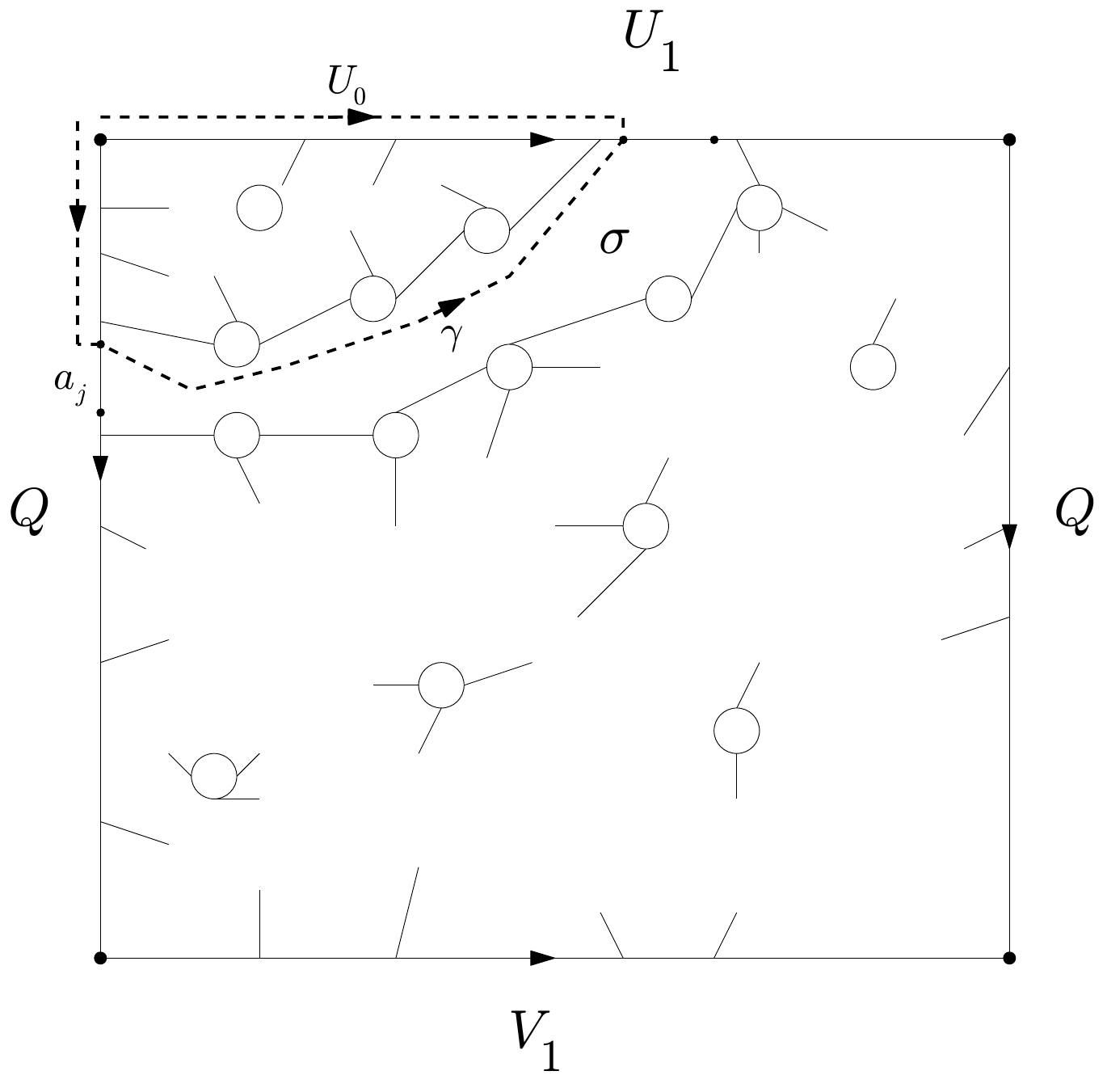}\hskip 0.8cm\includegraphics[scale=0.25]  {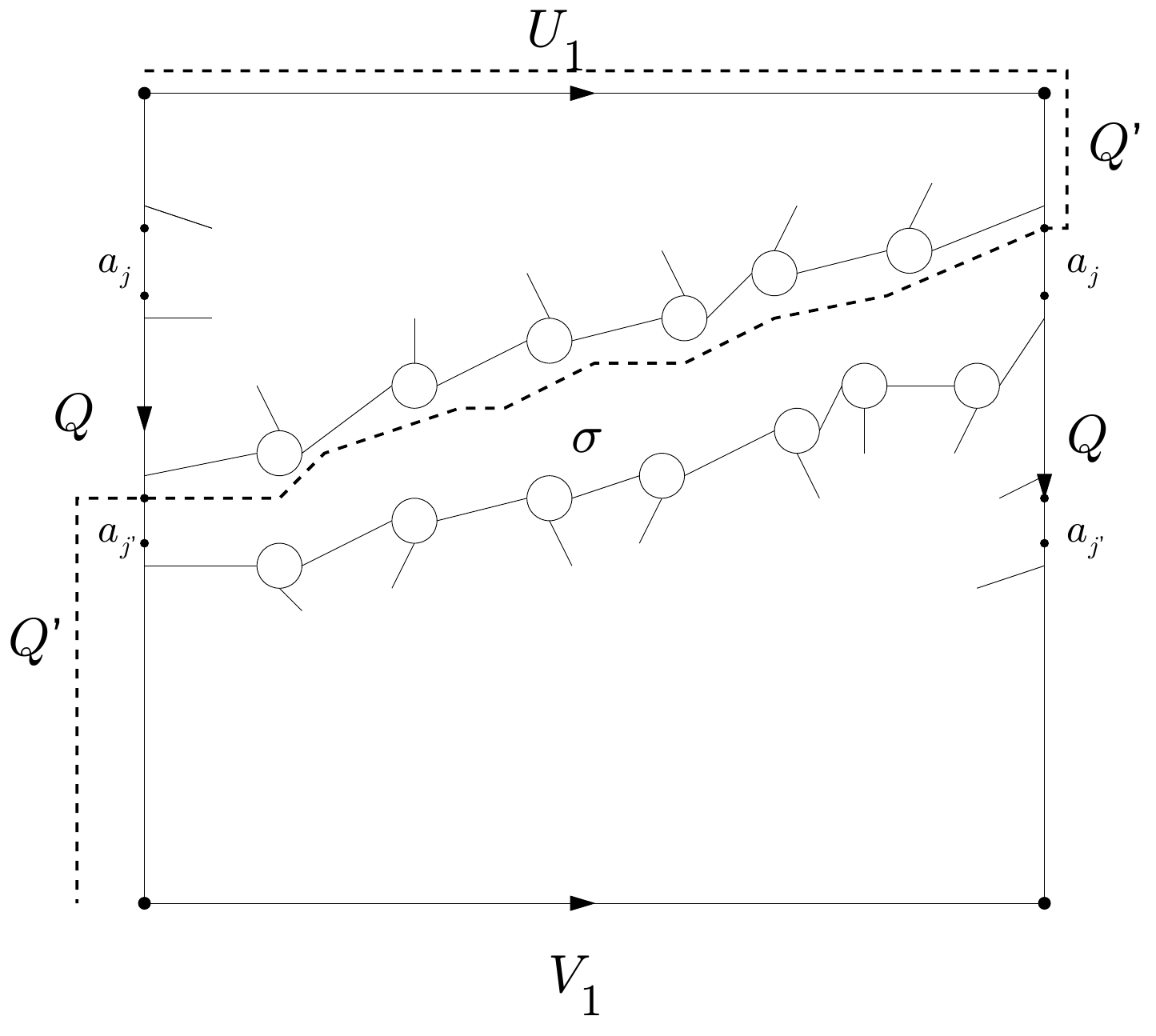}
	\caption{Two possible shortenings }\label{sigma1}
\end{figure}
		 
		\begin{figure}[h]
	\hskip -0.1cm\includegraphics[scale=0.3] {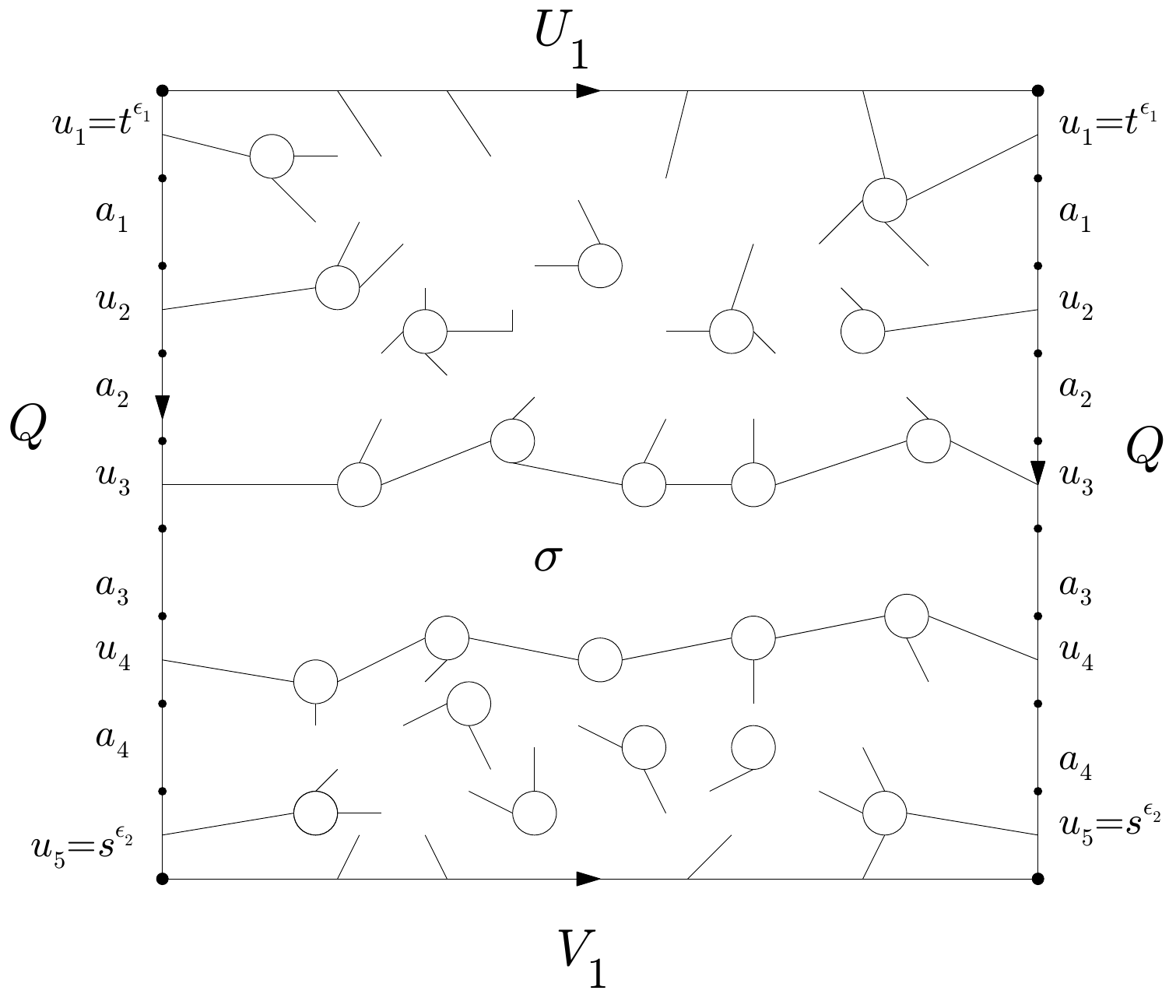}
	\caption{An $A$--region at $(3,3)$}\label{sigma2}
\end{figure}

		If $j<j'$ then $Q$ can be shortened, as there is a $(Y,Q')$-CCE where $Q':= t^{\epsilon_1}a_1u_1\dots a_{j-1}u_{j-1} w a_{j'}u_{j'}\dots a_\ell s^{\epsilon_2}$
		and $w$ is the label on the segments of the boundaries of the $\alpha$ discs in the top boundary of $\sigma$. 
		(See Figure \ref{sigma1}.)
		Similarly it is not possible to have $j>j'$.
		
	The remaining case is that $j=j'$, i.e. $\sigma$ meets both vertical sides in the segment labelled $a_j$.
	(See Figure \ref{sigma2} where $j=3$.)
	
	Consider the closed paths 
$U(p)=U_1U_2^p, V(p)=V_1V_2^p$ for $p>0$: 
	there are $(Y,Q)$-CCEs of the form $(U(p),U(q),V(p),V(q))$ when $p\neq q$
	(here $U(p),U(q)$ generate a rank 2 free subgroup of the subgroup generated by $U_1,U_2$).
	If for infinitely many choices of $(p,q)$, both    pictures have $A$--regions, then
	there is a pair $(p,q)$ such that the pictures $P_1',P_2'$ 
	for $U(p)\sim QV(p)Q^{-1}$ and $U(q)\sim QV(q)Q^{-1}$ contain $A$ regions at the same $a_j$,
	and there is a shorter $(Y,Q')$-CCE, with
	$Q':=t^{\epsilon_1}a_1u_1\dots a_{j-1}u_{j-1}$.
	
	We have thus shown that (for infinitely many choices of $p,q$) there
	are connected pictures for CCEs for $Y$.
	\end{proof}

\smallskip 
Recall that for the relator $R=u_1x_1u_2x_2\dots x_k$, 
	we can assume that the edges of the path $\partial\alpha$ alternate between an edge
	in $\cS\cup\cT$ and an edge  
from a component of $X$. 

Suppose that $P$ is one of the connected pictures in Lemma \ref{Agenerates}, and $\sigma$ is a region of $P$ that meets the boundary in a segment $b$ between two $\cS\cup\cT$-arcs.  Then, as discussed in the proof of Lemma \ref{Agenerates}, $\sigma$ is simply connected and represents a nullhomotopy of the label of $\partial\sigma$ in some component $X_\lambda^+$ of $X^+$.  
Now $b$ is part of $\partial\sigma$, labelled by some path in  $X_\lambda$.  
The rest of $\partial\sigma$ consists of $\cS\cup\cT$-arcs 
(which map to points in $X\cup\cS\cup\cT$), separated by half-edges of $\cS$- and $\cT$-edges, together with segments of the boundaries of $\a$-discs, each of which is labelled by one of the paths $x_j^{\pm1}$ that lie in $X_\lambda$.     
This motivates the following adjustment process.

For each component $X_\lambda$ of $X$, let $\Theta_\lambda$ denote the subgraph $X_\lambda\cap\partial\a$ of $X_\lambda^{(1)}$ and let $\Theta_{\lambda,1},\dots,\Theta_{\lambda,\ell(\lambda)}$ be its connected components.

	For each component $\Theta_{\lambda,i}$,  let $A_{\lambda,i}$ denote the image of $\pi_1(\Theta_{\lambda,i},v_{\lambda,i})$ in 
 $\pi_1(X_\lambda,v_{\lambda,i})$ (for some choice of base-point $v_{\lambda,i}$ in $\Theta_{\lambda,i}$). 
	 Then, letting $X_\lambda^+$  be the component $X_\lambda$ together with half edges meeting it, 
	 as in Lemma \ref{Agenerates},  let   ${X_{\lambda,i}}^+$ be a copy of the connected covering of   $X_\lambda^+$ with fundamental group $A_{\lambda,i}$. 	 
	 Let $X'$ denote the disjoint union of the ${X_{\lambda,i}}^+$, and note that $X'$ comes with a natural projection $\pi:X'\to X^+$ made up of the covering projections ${X_{\lambda,i}}^+\to {X_\lambda}^+$.
	
\begin{lem}[Adjustment $\Theta$]\label{Theta} 
Let $(U_1,U_2,V_1,V_2)$ be a $(Y,Q)$-CCE such that
\begin{itemize}
\item $Y$ has least possible complexity; and
\item the nullhomotopies of $QV_iQ^{-1}U_i^{-1}$ ($i=1,2$) are represented by connected pictures $P_1,P_2$.
\end{itemize}
Let $\pi:X'\to X^+$ be as above. 
Then we can choose, for each edge $u\in\cS\cup\cT\subset X$, a preimage in $X'$ of each of its half-edges, then identify the loose ends of these half-edges to form an edge $u'$.  We may then
 add a $2$-cell $\a'$ to  the complex $X''$ resulting from $X'$ via these identifications, in such a way that $\pi$ extends to a projection $\pi:Y'=X''\cup\a'\to Y$ sending $u'$ to $u$ for each $u$ and $\a'$ to $\a$.  Moreover $c(Y')\le c(Y)$, and the nullhomotopies of $QV_iQ^{-1}U_i^{-1}$ factor through $\pi$.
\end{lem}

\begin{proof}
From the construction, each $\Theta_{\lambda,i}$ in $X$ has an isomorphic copy $\Theta'_{\lambda,i}$ in $X'$, that projects isomorphically onto $\Theta_{\lambda,i}$ via $\pi$.  
 For each $u\in \cS\cup\cT$, let $u_+$ and $u_-$ denote its two half-edges.  
 The half-edge $u_+$ is attached to a vertex $v_+$ of some $\Theta_{\lambda,i}$, and $u_-$ is attached to a vertex $v_-$ of some $\Theta_{\mu,j}$.   Let $u'_+$ and $u'_-$ denote the unique half-edges in $\pi^{-1}(u_+), \pi^{-1}(u_-)\subset X'$ respectively that are connected to the vertices $v'_+:=\pi^{-1}(v_+)$ of $\Theta'_{\lambda,i}$ and $v'_-:=\pi^{-1}(v_-)$ of $\Theta'_{\mu,j}$ respectively.   We identify the loose ends of $u'_+$ and $u'_-$ to form an edge $u'$ and extend the range of $\pi$ by defining $\pi(u'):=u$. Performing this move for all $u\in\cS\cup\cT$ extends $\pi$ to $X'\cup\{u';~u\in\cS\cup\cT\}$.

The boundary cycle $\partial\a$ is an alternating concatenation of $(\cS\cup\cT)^{\pm1}$-edges and edges $x_j$
in $\bigsqcup_{\lambda,i}\Theta_{\lambda,i}$.  This can be (uniquely) lifted to $X''$, where each $u^\epsilon\in(\cS\cup\cT)^{\pm1}$ lifts to ${u'}^\epsilon$ and each $x_j$ from $\Theta_{\lambda,i}$ lifts to $\Theta'_{\lambda,i}$ via the isomorphism $\Theta'_{\lambda,i}\to\Theta_{\lambda,i}$.  Hence we can (uniquely) attach a $2$-cell $\a'$ so as to extend the projection $\pi$ by defining $\pi(\a'):=\a$.   Since $\partial\a'$ has the same length as $\partial\a$ and visits at least as many components of $X'$ as those of $X$ visited by $\partial\a$, it follows that $c(Y')\le c(Y)$, as claimed.

Now consider the relative rectangular picture $P_i$ that represents the nullhomotopy of $QV_iQ^{-1}U_i^{-1}$.  Each $\a$-disc in $P_i$ maps to $\a$ (possibly with a switch of orientation), and this map can be lifted to $\a'$.   
A small regular neighbourhood of each $u$-arc ($u\in\cS\cup\cT$) maps onto the $1$-cell $u$, so this mapping lifts to $u'$.  
Since $P_1$ is connected, each of its regions $\sigma$ is simply-connected, and meets the boundary of the rectangle in at most a single segment 
-- with the rest of $\sigma$ mapping to some $X_\lambda^+$, and meeting precisely
one $\Theta_{\lambda,i}$.  
The homotopy lifting property of coverings then allows us to lift the map on $\sigma$ to ${X_{\lambda,i}}^+$.   Putting all these lifts together gives the required lift of the whole nullhomotopy to $Y'$.
\end{proof}

We aim to prove Theorem \ref{main} by induction on complexity.  In what follows we will assume that we have made the adjustment indicated by Lemma \ref{Theta}; in other words that each $\Theta_\lambda:=X_\lambda\cap\partial\a$ is connected and $\pi_1$-surjects onto $\pi_1(X_\lambda)$.
This will allow us to factor our nullhomotopies through a $\Z$-cover defined by an epimorphism $\Psi:\pi_1(Y)\to\Z$, with a view to creating a lower-complexity CCE.  This applies only for the inductive step of the procedure; first we must address the initial case.

\bigskip
\noindent
Initial Case for the induction:

By hypothesis, each 1-cell $u \in\cS\cup\cT$ occurs in $\partial\a$. 
Moreover, if $u$ separates $Y\setminus\a$ then it must occur at least  
once with each orientation in $\partial\a$. 
Since $Y$ is connected  and $|\cS\cup\cT|>1$, it follows that the least
possible complexity is $c(Y)=k-|\pi_0(X)| = 0$, 
which is realised only when $\vert\cS\cup\cT\vert=k = \vert\pi_0 (X)\vert$,
the components of $X$ and the connecting 1-cells $\cS\cup\cT$ 
	are arranged in a cyclic manner, with each $u\in\cS\cup\cT$ appearing exactly once in $\partial\a$.
So the inductive proof starts with precisely that case.

As each $\cS,\cT$ edge occurs exactly once in $R$, in a reduced picture on any surface over $Y$
there are no $\cS,\cT$ edges joining $\alpha$--discs.
It follows that on a reduced annular picture, for each $\a$ disc, all $\cS$ edges go to the top boundary,
and all $\cT$ edges to the bottom boundary.  
Thus $Q$  is homotopic rel endpoints in $Y$ to $Q_1\cdot Q_2$ with  $Q_1,Q_2$ paths in $X\cup\cS$ and $X\cup\cT$ respectively,
and the Collins property holds.

\bigskip

\noindent
Inductive Step

\smallskip
For the inductive step of the proof we assume that the result holds
for  complexes of smaller complexity  satisfying the conditions for the construction of $Y$ 
at the beginning of section 2. 
Adjustment $\Theta$  (Lemma \ref{Theta}) enables us to construct a suitable $\Z$-cover $p :\wh Y\to Y$
corresponding to an epimorphism $\pi_1(Y) \to\Z$ as follows. 
In the cover  $\wh Y$ let $\wh \cS,\wh \cT$ denote the lifts of $\cS,\cT$.

As $|\pi_0(X)| - |\cS\cup\cT|=\chi(\mathcal G)$ is the Euler characteristic of the graph $\mathcal G$
having a vertex for each 
component of $X$ and an edge for each edge in $\cS\cup\cT$, and this graph is connected, so $\chi(\mathcal G)\le 1$.
Thus either $\chi({\mathcal G}) <0$ (case 1 below) or $\mathcal G$ is a tree (included in case 2 below)
or $|\cS\cup\cT|=|\pi_0(X)|$ and $\mathcal G$ is a circuit with 0, 1 or more trees attached (cases 4, 3, and 2 below respectively).
Thus at
least one of the following is true:
 
\smallskip\noindent
(1) $\vert\cS\cup\cT\vert> \vert\pi_0(X)\vert$: in this case there is an epimorphism 
$\Psi:\pi_1(Y)\to \Z$ that vanishes on the fundamental group of each component of $X$
(only the homology class of the loops matter).
Hence there is a corresponding $\Z$--cover $p:\wh Y\to Y$ such that the pre-image $\wh X$ of $X$ has the form $X\times\Z$.
Let $\wh \a$ be a 2-cell in $p^{-1}(\a)$. 
We claim that $\partial\wh \a$ passes through more than $\vert\pi_0(X)\vert$
components of $\wh X$. 
If not, let $Y'\subset\wh Y$ be the subcomplex consisting of $\wh \a$, the 1--cells of $\wh\cS\cup\wh\cT$
that occur in $\partial \wh\a$, and the components of $\wh X$ 
that meet $\partial\wh \a$.
It follows that  each component of $X$
has precisely one pre-image in $Y'$, and as 
$p$ restricts to an immersion $p|_{Y'} : Y' \to Y$
 it follows that each 1-cell in $\cS\cup\cT$
also has precisely one pre-image in $Y'$. 
It further follows that   
the immersion $p|_{Y'} : Y' \to Y$ is a cellular isomorphism, which is
absurd   as it factors through $\wh Y$.

\smallskip
\noindent
(2) There are two components $X_1,X_2$ of $X$, each of which meets
precisely one of the 1-cells of $\cS\cup\cT$. 
In this case   
$\pi_1(X_1)$ and $\pi_1(X_2)$ are non-trivial,  
finitely generated by the $x_j$ loops after adjustment $\Theta$
 (cf. Lemma \ref{Theta}).

Therefore there exists an epimorphism $\Psi: \pi_1(Y)\to\Z$ which
vanishes on the fundamental group of each component of $X$ except 
possibly for $X_1,X_2$ , but does not vanish on at least one of
these -- without loss of generality on $\pi_1(X_1)$. 
Let $p : \wh Y \to Y$ denote the $\Z$-cover corresponding to $\Psi$.
Let $X_0$ be the component of $X$ that is joined to $X_1$ by $u\in\cS\cup\cT$. 
Then $X_0 \neq X_2$ since $\vert\cS\cup\cT \vert>1$. 
Moreover there is a subpath of $\partial\alpha$ of the form 
$u^{-\epsilon} x_j u^\epsilon$ where   
$x_j$ is a  loop 
in $X_1$ with $\Psi(x_j)\neq 1$. It follows that any
lift of $u^{-\epsilon} x_j u^\epsilon $ to $\wh Y$ joins two distinct components of $p^{-1} (X_0 )$.

\smallskip\noindent
(3) $\vert\cS\cup\cT\vert=\vert\pi_0(X)\vert$ and there is exactly one component $X_1$ of $X$ that meets
only one $u \in\cS\cup\cT$ (joining $X_1$ to $X_0$ say). 
Then as before, $\pi_1(X_1)$ is not trivial.
There is also at least
one $u_0\in\cS\cup\cT$ which is non-separating in $Y \setminus\a$ (In particular, $u_0\neq u$.)
If $u_0$ appears in $\partial\a$ with exponent sum 0, then there is
a $\Z$-cover $p : \wh Y \to  Y$ such that the pre-image of $X$ has the form
$X \times \Z$ --- and we can argue as in the first case above. 
Otherwise there is an epimorphism $\Psi:\pi_1(Y)\to\Z$  which vanishes on the
fundamental group of each component of $X$ other than $X_1$ , but
not on $\pi_1(X_1)$. Then we can argue as in the second case above. 

\smallskip
\noindent (4) $\vert\cS\cup\cT\vert = \vert\pi_0(X)\vert$ 
and the components of $X$ form a cycle $\mathcal C$, 
each component meeting  precisely two of the 1-cells $\cS\cup\cT$. In particular
each of these 1-cells is non-separating in $Y\setminus\a$, and they all
appear with the same exponent-sum (in absolute value) in $\partial\a$.
If the exponent-sum is zero, then there is a $\Z$-cover $p : \wh Y\to  Y$
such that $\wh X$ 
has the form $X \times \Z$, and we can
argue as in the first case. Otherwise, for each component $X_i$ of
$X$ there is at least one
subpath $u_j x_j u_{j+1}$ of $\partial\a$ such that $x_j$ is an edge 
in  $X_i$ and $u_{j+1} \neq u_j^{-1}$. 

If $\pi_1(X_i)$ is trivial for every component $X_i$, then $\pi_1(X\cup\cS\cup\cT)\cong\Z$,
so $\pi_1(Y)$ has 
no non-cyclic subgroups.   

 Suppose without loss of generality that $\pi_1(X_1)\ne 1$.
After adjustment $\Theta$ we have that $\pi_1(X_1)$ is generated by the edges in $\Theta_1$ 
	outside of a maximal tree. 
	In fact the maximal tree  has at most one edge, as there at most two vertices in each $\Theta_j$, the endpoints of the two edges of $\mathcal C$ meeting $\Theta_j$. 
Without loss of generality suppose that $x_1$ is an edge  
in $X_1$ and that, 
 in the subpath $u_1x_1u_2$ of $R$, $u_1$ (resp $u_2$) is an edge in $\cS\cup\cT$ from $X_0$ to $X_1$ (resp $X_1$ to $X_2$).  

 Note that $u_1\ne u_2^{-1}$ since $|\cS\cup\cT|\ge 2$. Hence $X_0\ne X_1\ne X_2$ (but possibly $X_0=X_2$). Note also that the endpoints $v_1$ of $u_1$ and $v_2$ of $u_2$ that lie in $X_1$ are pairwise distinct by the conditions in the Reformulation from \S \ref{refandreds}. The edge $x_1$ of $X_1$ joins $v_1$ to $v_2$ and we can choose a maximal tree $T$ in $\Theta_1$ consisting of the edge $x_1$ and ts endpoints $v_1,v_2$.
Via an epimorphism  $\Psi:\pi_1(X_1)\onto \Z$ there is also a $\Z$-cover $p : \wh Y\to Y$
such that the pre-image of any component $X_i$ of $X$ other than
$X_1$ has the form $X_i\times\Z$, but the pre-image of $X_1$ does not. 

Using the maximal tree $T$ in $\Theta_1$, we see that $\pi_1(X_1,v_1)$ is generated by closed paths of three kinds:
\begin{enumerate}
\item loops $x_j$ at $v_1$;
\item paths $x_1 x_j x_1^{-1}$ where $x_j$ is a loop at $v_2$; and
\item paths $x_1 x_j^\varepsilon$ $(\varepsilon=\pm1)$ where $x_j$ is an edge between $v_1$ and $v_2$.
\end{enumerate} 
Now at least one of these closed paths represents an element of $\pi_1(X_1)\setminus\ker\Psi$.  In the first case $u_1x_ju_1^{-1}$ is a subpath of $\partial\a$, any lift of which to $\wh Y$ joins distinct components of $X_0$.
In the second case $u_2^{-1}x_ju_2$ is a subpath of $\partial\a$, any lift of which to $\wh Y$ joins distinct components of $X_2$.  In the third case $u_1x_1u_2$ and $u_1x_j^{-\varepsilon}u_2$ are subpaths of $(\partial\a)^{\pm1}$, any lifts of which to $\wh Y$ with the same intial point in $p^{-1}(X_0)$ will end on distinct components of $p^{-1}(X_2)$.

In all cases, any lift of $\partial\a$ to the cover $\wh Y$ meets strictly more
than $\vert\pi_0(X)\vert$ components of $p^{-1}(X)$. 

\bigskip 
We now return to
apply the above observations to a $(Y,Q)$-CCE, say $(U_1,U_2,V_1,V_2)$,
where by Remark \ref{passtosubgroup} we can assume that $U_1,U_2$ 
represent elements of the commutator subgroup of $\pi_1(Y,v_0)$,
and so  lie in $\ker\Psi$.  
Hence we can 
 lift the pictures to $\wh Y$.  If
the lifted pictures involve only one lift $\wh\a$ of $\alpha$, 
then there is a $(\ol Y, \wh Q)$-CCE  in $\ol{Y}:=p^{-1}(X\cup\cS\cup\cT)\cup\wh\a$ with smaller complexity, and so the Collins 
property will hold by induction. 

 We split the rest of the proof into two cases:
\begin{enumerate}
\item There is a single preimage $\wh\a$ of $\a$ in $\wh Y$ whose boundary $\partial\wh\a$ contains all the $p^{-1}(\cS)$-cells in the image of the lift $\wh\iota_K:St(U_1,U_2)\to\wh Y$ of the immersion $\iota_K(St(U_1,U_2)\to Y$.  (See \S \ref{first}.)
\item There is no such $\wh\a$. (See \S\S \ref{restrict}, \ref{formW1} and \ref{finalcurtain}.)
\end{enumerate}

\section{First case}\label{first}

In this section we deal with the 
 first case of the inductive step in the proof of Theorem \ref{main}.
The action of $\Z$ on $\wh Y$ by deck-transformations 
gives rise to an indexing $\a_m$ ($m\in\Z$) of the preimages of $\a$ in $\wh{Y}$.  
In this first case
there is a lift $\a_m$ of  $\alpha$ 
in $\wh Y$ such that 
	all lifts of $\cS$ edges in the lift of the Stallings graph $\wh\iota_K(St(U_1,U_2))$
lie in $\partial\a_m$. 

\smallskip

We shall use the following result  
in $\wh Y$.  

\begin{lem}[Iterated Brodski\u\i\ Theorem]\label{itBr} 
	Let $\mathcal Y$ be the ascending union of a sequence (finite or infinite) of $2$-complexes $$Y_0\subset Y_1 \subset\cdots,$$ where each component of $Y_0$ has locally indicable  
(possibly trivial) fundamental group, and $$Y_{n+1}=Y_n\cup E_n\cup\a_n$$ 
	for some non-empty set $E_n$ of $1$-cells and $\a_n$ a $2$-cell.  
	Assume that $\partial\a_n$ is not freely homotopic in $Y_n\cup E_n$ to a proper power, nor to a path in $Y_n$.   
	Suppose that $Q$ is an edge-path in $\mathcal Y$ from a $0$-cell $v_0$ to a $0$-cell $v_1$, 
	and that $H_0<\pi_1(Y_0,v_0)$, 
	$H_1< \pi_1(Y_0,v_1)$ are non-cyclic subgroups such that $H_0=Q H_1 Q^{-1}$ in $\pi_1(\mathcal Y,v_0)$. 
	
	Then $Q$ is homotopic rel endpoints in $\mathcal Y$ to an edge-path $Q_0$ in $Y_0$
	 such that $H_0=Q_0 H_1 Q_0^{-1}$ in $\pi_1(Y_0,v_0)$.
\end{lem}

\begin{proof}
	By induction on $n$ using the Freiheitssatz, Theorem \ref{Freiheitssatz}, 
	and Theorem \ref{Gli}, each component of each $Y_n$ has locally indicable  (possibly trivial) fundamental group, 
	and each map $Y_n\hookrightarrow Y_{n+1}$ is $\pi_1$-injective.  
	So $\pi_1(\mathcal Y,v_0)$ is the ascending union of the $\pi_1(Y_n,v_0)$.  
	Moreover, $Q$ is a path in $Y_n$ for some $n$.  
	By the Freiheitssatz again, $H_0=Q H_1 Q^{-1}$ in $\pi_1(Y_n,v_0)$.  In particular, the result holds in the case $n=0$.

	Now suppose that $n\ge1$. Putting $\mathcal Y:=Y_n$, $X:=Y_{n-1}$, $E:=E_n$ and $\a:=\a_n$ in Lemma \ref{Br}, it follows that $Q$ is homotopic rel end points in $Y_n$ to an edge-path in $Y_{n-1}$; the result then follows from another induction on $n$. 
\end{proof}

The path $Q$ in the 1--skeleton of $Y$, from the vertex $v_0$
to the vertex $v_1$,  lifts to a path $\wh Q$ in $\wh Y$ from $\wh v_0$ to a vertex $\wh v_1$. 
The  annular pictures $P_i$ for $Q^{-1}U_iQ\sim V_i$  over $Y$ lift to
pictures $\wh{P}_i$ over $\wh{Y}$ for $i=1,2$.

Let $Y'$ be the connected subcomplex of $\wh Y$ 
consisting of all the lifts of $\a$ and of the $1$-cells in $\cS\cup\cT$ that appear in the lifted pictures  
$\wh P_1,\wh P_2$, together with all the components of the pre-image of $X$ in $\wh{Y}$ that meet these  lifted pictures.
This contains:
\begin{enumerate}  
	
	\item[$\bullet$] the lifts $\wh\iota_K(St(K))$  and  $\wh\iota_L(St(L))$ 
	 to $\wh Y$ based at $\wh v_0$ and $\wh v_1$ of the immersed Stallings graphs in $Y$;
	
	\item[$\bullet$] the lift $\wh Q$ of $Q$ from  $\wh v_0$ to $\wh v_1$;

	\item[$\bullet$] finitely many $\Z$-covers and/or copies of the components $X_j$;
	
	\item[$\bullet$] finitely many lifts of each $\cS,\cT$ edge;
	
	\item[$\bullet$] finitely many lifts $\alpha_{m_1} , \dots ,\alpha_{m_j}, \dots,\alpha_{m_N}$ of $\alpha$.
	
\end{enumerate}

Note that the pictures $\wh P_1,\wh P_2$ over $\wh Y$ are a generalised form of relative pictures as 
they are in general relative to several 
lifts of $\alpha$.
Suppose that the  lifts $\alpha_{m_1} , \dots ,\alpha_{m_j}, \dots,\alpha_{m_N}$ of $\a$
appearing in $Y'$ are ordered by their indices.

Note that   $\alpha_0$, the lift of $\alpha$ at the base point $\wh v_0$ of $Y'$, 
which is the base point of $\wh  Y$,
may or may not appear in this list. 
We suppose that there is an index $m=m_\mu$ such that all the $\wh\cS$ edges in 
$\wh \iota_K(St(K))$ lie in the boundary of  $\alpha_{m}$.

Let $Y_0\subset Y'$ be the subcomplex of all the lifts of $X$ components in $Y'$, 
all $\wh\cS$ edges in $\wh \iota_K(St(K))$ and all  $\wh \cT$ edges in $\wh\iota_L(St(L))$
together with all $\wh\cS,\wh\cT$ edges in $\partial\alpha_m$, and the 2-cell $\alpha_m$. 

We wish to apply  Lemma \ref{itBr}
with:
 
\noindent
$H_0:=(\wh\iota_K)_*(\pi_1(St(K),*_K))<\pi_1(Y_0,\wh v_0)$, \\
$H_1:=(\wh\iota_L)_*(\pi_1(St(L),*_L))<\pi_1(Y_0,\wh v_1)$, \\
where $*_K,*_L$ are the base-points of $St(K)$ and $St(L)$.

Also, if $m'\in[m_1,m_N]$ and $m'< m$ (resp. $m'>m$) then there is at least one $\wh\cS$ edge in $\partial\alpha_{m'}$
with index strictly less than (resp. strictly greater than) any index on a $\wh \cS$ edge in $\partial\alpha_m$,
and so does not appear in $\wh\iota_K(St(K))$. 
Suppose that $m=m_\mu$, where $m_1,m_2,\dots,m_\mu,\dots, m_N$ are the 
indices of the lifts of $\alpha$ in $Y'$.
In order to apply   Lemma \ref{itBr},
take:

for $j=1,..,\mu-1$:

$Y_j =Y_{j-1}\cup E_j\cup \alpha_{m_{\mu-j}}$ where $E_j$ is the set of $\wh\cS,\wh\cT$ edges 
in  $\alpha_{m_{\mu-j}}$ not already present in $Y_{j-1}$; 

so $Y_{\mu-1} = Y_{\mu-2}\cup E_{\mu-1}\cup\alpha_{m_1}$
contains $\alpha_{m_1},\dots,\alpha_{m_\mu}$;

$Y_{\mu} =Y_{\mu-1} \cup E_{\mu} \cup \alpha_{m_{\mu+1}}$ where $E_{\mu}$
is the set of $\wh\cS,\wh\cT$ edges in $\alpha_{m_{\mu+1}}$ not already present in $Y_{\mu-1}$

$\vdots$

$Y_{N-1}= Y_{N-2} \cup E_{N-1} \cup \alpha_{m_N}$ where $E_{N-1}$
is the set of $\wh\cS,\wh\cT$ edges in $\alpha_{m_N}$ not already present in $Y_{N-2}$;
thus $Y_{N-1}=Y'$.

At each stage, $\wh\cS$ and $\wh\cT$ edges are added, and the
choice of $m_\mu$ ensures that every time the set of $\wh\cS$ edges added is non--empty.

\smallskip 
Applying Lemma \ref{itBr} gives the conclusion that in $Y_0$, 
the path  $\wh Q$ is homotopic rel. endpoints
to a path $\wh Q_0$  such that   $H_0=\wh Q_0H_1\wh Q_0$ in $\pi_1(Y_0,\wh v_0)$.
But $Y_0$ contains just one lift of $\alpha$, and thus the  Collins property holds
in $Y_0$ by induction on complexity. 
This implies that $\wh Q_0$  
is homotopic rel endpoints in $Y_0$ to $\wh Q_1\cdot \wh Q_2$,
 with $\wh Q_1$ in 
$\wh X\cup\wh \cS$ and $\wh Q_2$ in $\wh X\cup\wh\cT$,  
and thus in $Y$ there is a homotopy between $Q$ and a concatenation $Q_1\cdot Q_2$
with $Q_1$ a path in $X\cup\cS$ and $Q_2$ a path in $X\cup\cT$,
and the Collins property holds in $Y$.

 \section{Restricting the conjugating element}\label{restrict}
   
In this section we begin the second case of the inductive step in the proof of Theorem \ref{main}, as set out in \S \ref{plan}.
There is an epimorphism $\Psi:\pi_1(Y)\to\Z$ and a corresponding $\Z$-cover $p:\wh{Y}\to Y$ such that our immersions $\iota_K:St(U_1,U_2)\to Y$ and $\iota_L:St(V_1,V_2)\to Y$ lift to $\wh{Y}$, as do the pictures $P_1$ and $P_2$.
But the boundary of any given $2$-cell $\wh\a\in p^{-1}(\a)$  does not contain all of the $p^{-1}(\cS)$-arcs in the image of the lift of $\iota_K$.
As in \S \ref{first}, the deck-transformation action of $\Z$ on $\wh Y$ gives rise to an indexing of cells, which we will also refer to as a {\em $\Z$-labelling}.  Thus $p^{-1}(\a)=\{\a_m,~m\in\Z\}$ and $p^{-1}(u)=\{u_m,~m\in\Z\}$ for $u\in\cS\cup\cT$.

Here we will show that under these hypotheses 
we can restrict the choice of conjugating path $Q$ to a single specific path $W_1$.  This restriction will be exploited in the subsequent sections to complete the proof.
 
 In \S \ref{ords} we described how to use a right-ordering on $G:=\pi_1(Y)$ to identify 
 arcs $u_{min}$ and $u_{max}$ incident at any $\a$-disc in a picture, and to define a pre-order on the $\a$-discs in a rectangular picture.  
 For this purpose we choose a right ordering  $<$ that 
 is {\em dominated} by the natural ordering of $\Z$ via the epimorphism 
 $\Psi:\pi_1(Y)\onto\Z$ that defines the $\Z$-cover $p:\wh{Y}\to Y$, 
 in the sense that $\Psi(g)<\Psi(h)\Rightarrow g<h$.  
 We will fix that choice from now on   
 (except that at  a certain point in \S \ref{formW1} we will also consider the opposite ordering $>$).   
 This choice of right ordering has some useful consequences, as follows. 
 
 \begin{itemize}
 	\item If $\a_m$ is one of the lifts of $\a$ to $\wh Y$, $u\in\cS\cup\cT$, and $n,N$ are the least and greatest integers $j$  such that $u_j$ is involved in $\a_m$, then the $u_{min}$ and $u_{max}$ arcs 
 	of $\partial\a$ lift to $u_n,u_N$ arcs respectively of $\partial\a_m$.
 	\item  If $\b_1,\b_2$ are $\a$-discs in a rectangular picture $P$, 
 	$\wh P$ is a lift of $P$ to $\wh Y$ such that $\b_1,\b_2$ lift respectively to an $\a_m$-disc and an $\a_n$-disc in $\wh P$, 
 	with $m<n$, then $\b_1<\b_2$ in the pre-order on $\a$-discs of $P$.
 \end{itemize}

 \medskip\noindent{\bf Definition}:  Let $P$ be a reduced rectangular or annular picture.  
 An $\a$ disc  $\beta$ in $P$ is called an {\em up-down connection} 
 if there are $(\cS\cup\cT)$-arcs connecting it to each of the top and bottom boundaries (of the rectangle or annulus).  
 It is a {\em  $G$-min  (resp.  $G$-max) up-down connection} 
 if for each $s \in \cS$, the  $s_{min}$  (resp. $s_{max}$) arc of $\b$ ends on the top boundary of $P$,
 and for each $t \in \cT$, the  $t_{min}$ (resp.   $t_{max}$)  arc of $\b$ ends on the bottom boundary of $P$.
  It is possible for a $G$-max connection to be at the same time  a $G$-min connection.

 \begin{lem}	\label{surlepont}
 	Let $P$ be a reduced annular relative picture over $\wh Y$, 
 	such that the $\a$-discs of $P$ that are joined by $\cS$-arcs 
 	to the top boundary have two or more distinct $\Z$-labels.  
 	Then $P$ has at least two up-down connections, one $G$-max and one   $G$-min.
 \end{lem}
 
 \begin{proof}
 	Suppose that $\b_0,\b_1$ are two $\a$-discs of $P$ that are attached to the top boundary by $\cS$-arcs, and have $\Z$-labels $n_0,n_1$ respectively with $n_0<n_1$.   
 	If $\b_1$ is not a $G$-max up-down connection, 
 	then for some $u\in\cS\cup\cT$  its $u_{max}$-arc connects $\b_1$ to another $\a$-disc -- say $\b_2$.    
 	Iterating this process gives a  chain 
  $\b_1,\b_2,\cdots$ of $\a$-discs in $P$,  
 	where each $\b_j$ is joined to $\b_{j+1}$ by the  $u_{max}$-arc  of $\b_j$ for some $u\in\cS\cup\cT$.   
 	The chain ends in a  $G$-max up-down connection, 
 	or it contains a repetition -- say $\b_i=\b_j$ with $1\le i<j$. 
 	In the latter case, 
 		assuming that $\b_i=\b_j$ is the first repetition in the chain,
 	the path $\b_i$ -- $\b_{i+1}$ -- $\cdots$ -- $\b_j=\b_i$ 
 	 is simple, and cannot be nullhomotopic in 
 	 the ambient annulus, 
 	else there  would be a 
 	 disc-subpicture containing the path and the pre-ordering in this  disc-picture
 	would give a strictly $G$-increasing chain of $\a$ discs $\b_i<\b_{i+1}<\cdots<\b_j$ which 
 		could not be closed.   
 	 Hence
 	this path
 	 $\b_i$ -- $\b_{i+1}$ -- $\cdots$ -- $\b_j=\b_i$
 		wraps once around the annulus and cuts it 
 	into two smaller annuli.  
 	Then the path $\b_i$ -- $\b_{i-1}$ -- $\cdots$ -- $\b_1$ -- [top boundary] 
	splits the upper small annulus into a disc --  $D$ say.
 	
 	Now repeat this argument in the other direction from $\b_0$ to construct a 
 	chain $\b_0$, $\b_{-1}$, \dots with 
 	$\b_{-j}$ joined to $\b_{-j-1}$ 
 	by the $u_{min}$-arc of $\b_{-j}$ for some $u\in\cS\cup\cT$.  
 	The two chains cannot meet as the $\Z$ labels in the first are all strictly 
 	greater than the $\Z$ labels in the second
  -- by our choice of right-ordering $<$.
 	The second chain is thus entirely contained in the disc $D$, 
 	so cannot contain a repetition or an up-down connection, so continues indefinitely, 
 	contradicting the fact that $P$ has only finitely many $\a$-discs.  
 	Hence the first chain $\b_1,\b_2,\cdots$ above must end in a $G$-max up-down connection.   
 	For similar reasons, the second chain  $\b_0,\b_{-1}\cdots$ must end in a  
 	$G$-min up-down connection.
 \end{proof}

 	\begin{cor}\label{everyedge}
 		If no $2$-cell of $\wh Y$ contains every $\cS$-cell of $\wh{\iota}_K(St(K))$ in its boundary, then there is a rectangular picture $P^{rect}$ with both vertical sides labelled $Q$ that contains a $G$-min up-down connection and a $G$-max up-down connection.
 	\end{cor}
 
 \begin{proof}
 	Since $St(K)$ is a  finite core graph, it contains a cyclically reduced path containing all its edges.
 	Applying $\wh\iota_K$ to such a path gives a path $\wh{U}$ in $\wh{X}\cup\wh{\cS}$.
 	 We can write $\wh U$  as a word in the generators $U_1,U_2$ of $K$, 
 	 then combine copies of the homotopies $Q^{-1}U_iQ\sim V_i$ in  Theorem \ref{2ComplexVersion} to obtain a homotopy  $Q^{-1}\wh{U}Q\sim \wh{V}$ which can be represented by a rectangular picture $P^{rect}$.  
 	 The result then follows from Lemma \ref{surlepont} applied to the annular picture $P^{ann}$ obtained from $P^{rect}$ by identifying its vertical sides and reducing via disc-cancellations. (Note that $\wh U$ remains strongly cyclically reduced under such cancellations.)
 \end{proof}
 
   From the picture $P^{rect}$ of Corollary \ref{everyedge}
 we can excise any  $G$-min or $G$-max up-down connection in the form of a rectangular sub-picture consisting of a single $\a$-disc and its incident arcs, which we call $\G_{min}$ or $\G_{max}$. 
  The top label  $U_{min}$ (resp. $U_{max}$) of this picture is an $\cS$-syllable of $R^{\pm 1}$ containing   the $s_{min}$ edges (resp. $s_{max}$ edges) 
 for all $s\in \cS$, and its bottom label $V_{min}$ (resp. $V_{max}$)
 is a $\cT$-syllable of $R^{\pm 1}$ containing
 the   $t_{min}$ edges (resp. $t_{max}$ edges) 
 for all $t\in \cT$.
  Note that these two special pictures 
 coincide if the corresponding syllables coincide.

 We denote the left side label of $\G_{min}$  
 by $W_1$ and its right side label by $W_2$, so that $W_1V_{min}W_2^{-1}U_{min}^{-1}$ is freely homotopic 
  in $X\cup\cS\cup\cT$ to $R^{\pm1}$.  The exponent depends on the choice of orientation of the up-down connection in $\G_{min}$, so that $W_1,W_2$ are interchangeable via the symmetry $R\leftrightarrow R^{-1}$.
 
 Similarly,  the top and bottom labels of $\G_{max}$ are $U_{max}$ and $V_{max}$, and we denote the left and right side labels by $W_3$ and $W_4$.  Again, $W_3$ and $W_4$ are interchangeable via a change of orientation of the $\a$-disc.

 Note that the labels $U_{min}, U_{max}, V_{min},V_{max}$ on the top and bottom edges being syllables of $R^{\pm1}$ 
 means that the paths   $W_i$ begin with  $\cT$-edges 
 and end with $\cS$-edges and  
 the syllable lengths satisfy  
$$\text{SL}(R) = \text{SL}(W_1) + \text{SL}(W_2) + 2 = \text{SL}(W_3) + \text{SL}(W_4) + 2.$$
 
 \begin{figure}[h] 
 	\hskip -0.5cm\includegraphics[scale=0.32]{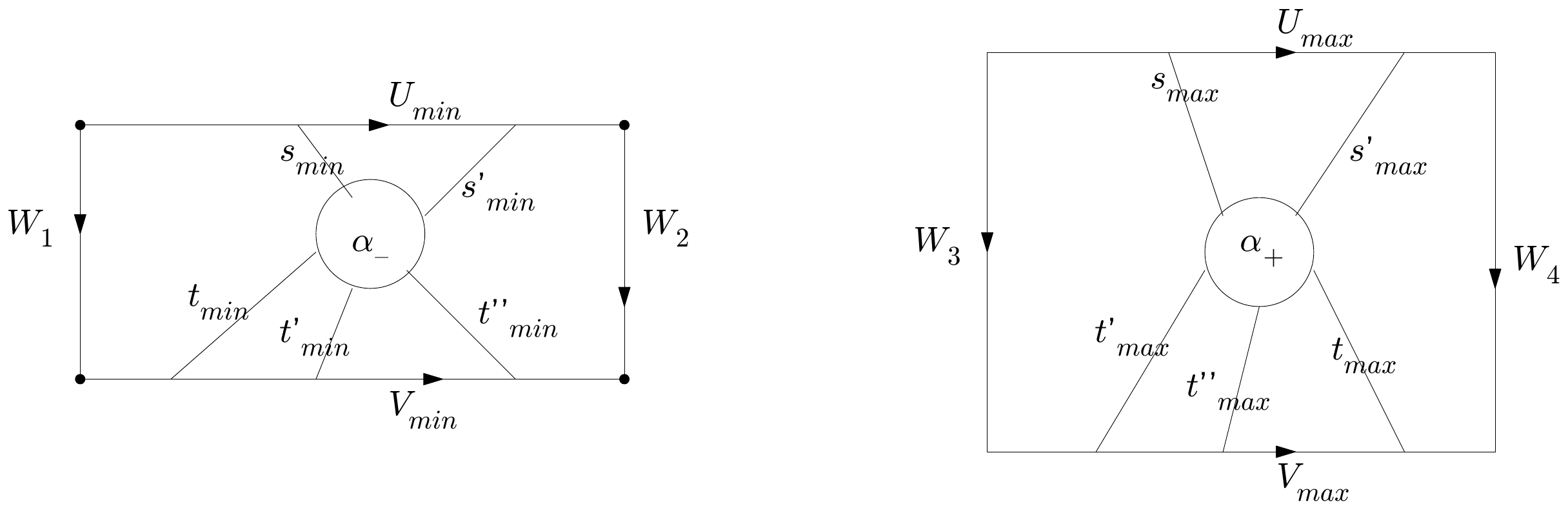}
 	\caption{The special  1--relator  rectangular { pictures} $\G_{min}$ and $\G_{max}$}\label{f1}
 \end{figure}
 
 Without loss of generality, we may assume that $W_1$ has syllable-length less than or equal to those of $W_2,W_3,W_4$:
 \begin{equation}\label{W1least}
 	\text{SL}(W_1)\le\min\{\text{SL}(W_2),\text{SL}(W_3),\text{SL}(W_4)\}.
 \end{equation}

 	Returning to 
 	Theorem \ref{2ComplexVersion}, we can now show that it can be reduced to the special case where $Q$ is any chosen one of the paths $W_1,\dots,W_4$ that label the sides of $\G_{min}$ and $\G_{max}$. 
 	
 	\begin{lem}
 		For each $j\in\{1,2,3,4\}$,  Theorem \ref{2ComplexVersion} is true if and only if it is true for $Q=W_j$.
 	\end{lem}

 \begin{proof}
 		By symmetry it suffices to prove this for $W_1$.
 		Cut  the picture  $P^{ann}$ from Corollary \ref{everyedge} along  a path traversing a  $G$-min  up-down connection along 
 	the side that contains a lift $\wh W_1$ of $W_1$ from $Y$ to $\wh Y$. This 
 	gives a rectangular picture  with 
 	vertical side-labels $\wh W_1$.
 	Moreover $\wh Q$ is homotopic rel. endpoints to  $\gamma \wh W_1\delta$ with $\g$ a path in $\wh{X}\cup\wh{\cS}$ and $\d$ in $\wh{X}\cup\wh{\cT}$.   
 	The statement of  Theorem \ref{2ComplexVersion} is equivalent to the same statement 
 	with $\wh{Q}$ replaced by $\wh W_1$, 
 	so we may assume that, in fact $\wh{Q}=\wh W_1$, and $\wh v_0, \wh v_1$ are its initial and final vertices
 	-- and indeed that the rectangular picture obtained from cutting $P^{ann}$ is the picture $P^{rect}$ that we began with.
 	
 	Projecting back down from $\wh Y$ to $Y$, we are reduced to consideration of the intersection of $\pi_1(X\cup\cS)$ and $W_1\cdot\pi_1(X\cup\cT)\cdot W_1^{-1}$ in $\pi_1(Y)$,  as claimed.
 \end{proof}

\section{The form of reduced rectangular pictures with vertical sides labelled $W_1$}\label{formW1}

In this section we continue the inductive step in the proof of Theorem \ref{main}.  
Having reduced ourselves to consideration of CCEs with $Q=W_1$ by the results in \S \ref{restrict},  
we shall now show that reduced rectangular pictures with both vertical sides labelled $W_1$
have a certain restricted form. 
This will be exploited in \S \ref{finalcurtain} to complete the proof.

\begin{lem}\label{l:3} 
	Assume that the inequality (\ref{W1least}) holds.
	Let $\G$ be a reduced relative rectangular   picture
	over $Y$ with a single $\a$-disc, 
	such that both vertical sides of $\G$ have label $W_1$.  
	Then $\G=\pm\G_{min}$ and $W_2=W_1$.
\end{lem}

\begin{proof}
	Let $U,V$ be the labels of the top and bottom sides of $\G$ respectively, so that $\partial\G$ is labelled 
	$W_1VW_1^{-1}U^{-1}$.
	
	The cyclic  subpaths	$W_1$   and $W_1^{-1}$ of $W_1VW_1^{-1}U^{-1}$ 
	 are  {\em uniquely positioned}, 
	in the sense that no other cyclic subpath   
	of $W_1VW_1^{-1}U^{-1}$ is homotopic rel endpoints to  $W_1^{\pm1}$  in $X\cup\cS\cup\cT$.
 
		To see this, suppose for example that some subpath $Z$ of $W_1VW_1^{-1}$ -- other than the obvious initial segment -- is homotopic rel endpoints to $W_1$ in $X\cup\cS\cup\cT$.  (The other three possible counterexamples are analogous.)   Then $Z$ cannot be contained in $W_1V$ since the last edge of $W_1$ is an $\cS$-edge and $V$ contains no $\cS$-edges.   Hence some initial segment $Z_0$ of $W^{-1}$ is a terminal segment of $Z$ containing at least one $\cS$ edge, and $Z_0$ is homotopic rel endpoints in $X\cup\cS\cup\cT$ to a terminal segment of $W_1$ of the same $(\cS\cup\cT)$-length. But the terminal segment of $W_1$ in question is $Z_0^{-1}$. 
		
		In particular $Z_0$ must contain an even number of $\cS\cup\cT$ edges, for otherwise the middle such edge would be equal to its own inverse.    Hence the two middle $\cS\cup\cT$ edges in $Z_0$ are inverse to each other, and are separated by a closed path in $X$ which is homotopic rel endpoints to its own inverse (and hence nullhomotopic since $\pi_1(X)$ has no $2$-torsion).  This contradicts the hypothesis that $W_1$ is a cyclic subpath of the strongly cyclically reduced path $R^{\pm1}$.

	It follows that the {cyclic subpath} 
	$W_1$ of $R$ matches up to either the left or the right vertical side of $\G$, 
	and hence that $\G=\pm\G_{min}$, as claimed. 
In particular the side-labels $W_1^{\pm1}$ of $\G$ are equal to those of $\G_{min}$, namely  $W_1$ and $W_2^{-1}$.   Hence $W_1=W_2$ as claimed.
\end{proof}

\begin{lem}\label{l:2}
	Assume that the inequality (\ref{W1least}) holds.
	Let $\G$ be a reduced relative rectangular  picture 
	over $Y$ with more than one $\a$-disc, 
	such that both vertical sides of $\G$ have label $W_1$.    
	Then $\G$ can be decomposed in the form $\G_1+\Delta+\G_2$, 
	where each of $\G_1,\G_2$ is a copy of either  $\pm\G_{min}$ or of $\pm\G_{max}$. 
	Moreover the $\a$-discs $\b_1,\b_2$ of $\G$ given by Lemma \ref{l:1}
	are precisely the $\a$-discs in $\G_1$ and $\G_2$.
	
\end{lem}

\begin{proof} Let $\b_1,\b_2$ be the $\a$ discs given by Lemma \ref{l:1}.
	Suppose first that no ($\cS\cup\cT$)-arc joins $\b_1$ to the right side of the picture.  Lemma \ref{l:1} gives a sequence of consecutive ($\cS\cup\cT$)-arcs from $\b_1$ to the boundary; these arcs and the parts of the boundary that separate them are labelled by a path $\g_1$ in $Y^{(1)}$ that is a cyclic subpath  of $R^{\pm1}$ containing either the $u_{min}$ edges of $R^{\pm1}$ for all $u\in\cS\cup\cT$ or the $u_{max}$ edges of $R^{\pm1}$ for all $u\in\cS\cup\cT$. It follows from the structure of the single-disc pictures $\G_{min}$ and $\G_{max}$, together with the inequality \eqref{W1least}, that such a path $\g_1$ has  syllable length $\mathrm{SL}(\g_1)\ge\text{SL}(W_1)+2$.  Moreover the first syllable of $\g_1$ is an $\cS$-syllable that contains either the $s_{min}$ edge or the $s_{max}$ edge of $R$ (for each $s\in\cS$), while its last syllable is a $\cT$-syllable that contains either the $t_{min}$ edge or the $t_{max}$ edge of $R$ (for each $t\in\cT$).
		By hypothesis none of these edges is the label of any $\cS\cup\cT$ arc from $\b_1$ to the right hand side of $\G$.   Since $\mathrm{SL}(\g_1)\ge \mathrm{SL}(W_1)+2=\mathrm{SL(U^{-1}W_1V)}$, it follows that $W_1$ is contained in $\g_1$ which in turn is contained in $U^{-1}W_1V$. In particular, $\text{SL}(\g_1)=\text{SL}(W_1)+2$.
	
	It follows that all the ($\cS\cup\cT$)-arcs from the left side of the picture are joined to $\b_1$, as are at least some of the $\cS$-arcs from the top boundary and at least some of the $\cT$-arcs from the bottom boundary. 
	Moreover, either all of the $u_{min}$ arcs at $\b_1$ or all of the $u_{max}$ arcs go to the top or bottom sides of $\G$, by the inequality \eqref{W1least}.     
	Then by cutting $\G$ along the right side of $\b_1$ we can decompose $\G$ as $\G_1+\G'$ for some $\G'$, where  $\G_1$ is a copy of $\pm\G_{min}$ or $\pm\G_{max}$ and $\b_1$ is the $\a$-disc in $\G_1$. 
	It also follows that no ($\cS\cup\cT$)-arc joins $\b_2$ to the left side. 
	Repeating the argument, we have $\G'=\Delta+\G_2$ for some $\D$, where $\G_2$ is a copy of $\pm\G_{min}$ or $\pm\G_{max}$ 
	-- and $\b_2$ is the $\a$-cell in  $\G_2$ as claimed. 
	
	Parallel arguments give the result in all cases where one of $\b_1,\b_2$ fails to be connected to one of the vertical sides by ($\cS\cup\cT$)-arcs.  So we may assume that each of $\b_1,\b_2$ is connected to each of the vertical sides.
	Hence also all $\cS$-arcs from the top of the picture are connected to one of the $\b_j$, and all the $\cT$ arcs on the bottom are connected to the other.

	Let $\b_3$ be a 
minimal $\a$-disc in $\Gamma$ (possibly $\b_1$ or $\b_2$). 
	Its $u_{min}$ arcs do not end on other $\a$ discs for all $u\in \cS\cup\cT$.
	
 Suppose first that $\b_1$ is the chosen minimal $\a$-disc.  
		By hypothesis there is sequence of consecutive arcs joining $\b_1$ to the boundary that is labelled by a subpath $\g_1$ of $R^{\pm 1}$ of syllable length $\ge\text{SL}(W_1)+2$ that contains all of the $s_{min}$ letters of $R^{\pm1}$ (for $s\in\cS$) in its first syllable, and all the $t_{min}$ letters ($t\in\cT$) in its last syllable.   By hypothesis at least some of these arcs must be connected to the right side of the picture, for otherwise $\g_1$ contains $W_1$ and every $\cS\cup\cT$ arc from the left side goes to $\b_1$ . In turn that would imply that no arc joined $\b_2$ to the left side, contrary to hypothesis.   
		
		Similarly, if $\b_2$ is minimal, it is joined to the boundary by a sequence of consecutive arcs labelled by a cyclic subpath $\g_2$ of $R^{\pm1}$ of syllable length $\ge\text{SL}(W_1)+2$.  Then at least some of the $u_{min}$ arcs for $u\in\cS\cup\cT$ meet the right side of $\G$. 
		
		Since $\g_1^{\pm1}$ and $\g_2^{\pm1}$ are disjoint cyclic subpaths -- each of syllable length $\ge \text{SL}(W_1)+2$ -- of the boundary label of $\G$, which has syllable length $2\text{SL}(W_1)+2$, these paths must together  cover the whole of the boundary label of $\G$, except possibly for subpaths of a single syllable (say $\g_{left}$ and $\g_{right}$) in each of the $W_1^{\pm1}$ subpaths that label the vertical sides of $\G$.   Since no single syllable can contain both $\cS$ letters and $\cT$ letters, it follows that if $\b_1\ne\b_3\ne\b_2$ then all the $s_{min}$ arcs (for $s\in\cS$) from $\b_3$ meet one vertical side of $\G$ and all the $t_{min}$ arcs (for $t\in\cT$) must meet the other side.

		In all three cases, at least some of the  $u_{min}$  arcs from the minimal $\a$-disc $\b_3$ meet the right side of $\G$. Hence  forming $\Gamma':= \Gamma +\G_{min}$ gives a picture in which the new right hand $\a$ disc
		is less than the 
minimal disc in $\Gamma$ and so the new picture $\G'$ is reduced. 
		
		Now each of the discs promised in Lemma \ref{l:1} meets the boundary of the picture in a consecutive sequence of arcs whose labels spell a cyclic subpath of $R^{\pm1}$ of length $\text{SL}(W_j)+2$ for $j\in\{1,2,3,4\}$. And this is at least $\text{SL}(W_1)+2$ by inequality \eqref{W1least}. But the paths that label the sequences of arcs from $\b_1$ and $\b_2$ to the boundary of $\G'$ are subpaths of $(U_{min}^{-1}W_1)^{\pm1}$
		and $(W_1V_{min}^{-1})^{\pm1}$ respectively, each of which has syllable length $\text{SL}(W_1)+1$.   And any non-empty sequence of consecutive boundary arcs in  $\G'$ on any other $\a$-disc of $\G$ is labelled by a subpath of the single syllable $\g_{left}$.  Hence the $\a$-disc in $\G'\setminus \G$ is the only one in $\G'$ that can satisfy the properties of the $\b_i$ in Lemma \ref{l:1}.   But Lemma \ref{l:1} says that at least two $\a$-discs in $\G'$ satisfy these properties,
		giving a contradiction.
\end{proof}

\begin{cor}\label{W1W2}
	Assume that the inequality (\ref{W1least}) holds and that $W_1=W_2$.   
	Then  $\pi_1(X\cup\cS,v_0)\cap W_1\pi_1(X\cup\cT,v_1)W_1^{-1}$ is cyclic.
\end{cor}

\begin{proof}
	$\G_{min}$ is a reduced relative  picture 
	with vertical side  labels  $W_1$.  
	We show by induction on the number $n$ of $\a$-discs that any reduced relative  picture  
	$\D$ with $n$ $\a$-discs and vertical side  labels $W_1$ has the form $\pm n\G_{min}$ (up to boundary surgery), where  
	$$n\G_{min}:=\G_{min}+\cdots+\G_{min}~~(n~\mathrm{terms}).$$
	Thus $\pi_1(X\cup\cS,v_0)\cap W_1\pi_1(X\cup\cT,v_1)W_1^{-1}$ is generated by the top label $U_{min}$ of $\G_{min}$.
	
	If $n=1$ then the result follows from Lemma \ref{l:3}, so assume that $n>1$.  
	By Lemma \ref{l:2} we can write  $\D  = \G_1+\D'+\G_2$ 
	and as pointed out in the proof of Lemma \ref{l:3}, the {cyclic subpaths} 
	$W_1$ and $W_1^{-1}$ of $R$ are uniquely positioned, 
	from which it follows that each of $\G_1,\G_2$ is  $\pm\G_{min}$ and  
	$\G_1+\D'=\pm (n-1)\G_{min}$ by the inductive hypothesis.   
	Since $\D$ is reduced, it follows that $\D=\pm n\G_{min}$, as claimed.
\end{proof}

\begin{lem}\label{newlemma:1} 
	Assume that inequality \eqref{W1least} holds and $W_1\neq W_2$. 
	Let $\Gamma$ be a  reduced relative annular picture over $\wh Y$ with 
	two or more $\Z$ labels on the $\a$-discs joined to the top. 
	Then   $W_3=W_1$ or $W_4=W_1$,  and any $G$-min up-down connection in $\G$ is joined 
	by all the $\wh\cS\cup\wh\cT$-arcs in its $\wh{W}_1$ side to a $G$-max up-down connection.

\end{lem}

\begin{proof}
	Lemma \ref{surlepont}
says there is a $G$-min up-down connection $\b$ in $\G$.
	Splitting $\G$ along the $\wh W_1$ side of this gives a reduced rectangular picture $\Gamma^{rect}$ 
	for $U'\sim\wh W_1V'\wh W_1^{-1}$ for 
	(strongly after strong reduction) reduced paths $U'$ in $\wh X\cup\wh\cS$ and $V'$ in $\wh X\cup\wh\cT$.
 By Lemma \ref{l:2}, this picture has the form $\wh{\G}_1 + \wh{\Gamma}'+\wh{\G}_2$ where each of 
$\wh{\G}_i$ $i=1,2$ is a lift of $\pm\G_{min}$ or of $\pm\G_{max}$. 
 Indeed by construction the $\a$-disc at the left of $\Gamma^{rect}$
 (after perhaps replacing $\Gamma^{rect}$ by $-\Gamma^{rect}$) 
  is a $G$-min up-down connection, so necessarily $\wh{\G}_1$ is a lift of $+\G_{min}$.
 Now $\wh\G_2$ cannot be a lift of $+\G_{min}$ (since $W_1\ne W_2$), nor of $-\G_{min}$ (since $\Gamma$ is reduced). 
Hence $\wh{\G}_2$ is a lift of $\pm\G_{max}$, and one of $W_3$ or $W_4$ is equal to $W_1$ as claimed.  
 Furthermore, since $\G$ can be recovered from $\G^{rect}$ by identification of its vertical sides, 
 the $\a$-disc in $\wh{\G}_1$ is joined in $\G$ to the $\a$-disc in $\wh{\G}_2$ 
 by all the $(\wh\cS\cup\wh\cT)$-arcs in its $\wh W_1$ side.  
 But the former is the $G$-min up-down connection under consideration, 
 and the latter is a $G$-max up-down connection. The result follows.
\end{proof}

In particular, at least one of $W_3,W_4$ also has least syllable length among the $W_j$. 
Applying Lemma \ref{W1W2} again with respect to the opposite right order $>$ on $G$, 
we may also assume that $W_3\ne W_4$.  Without loss of generality
 (replacing $\G_{max}$ by $-\G_{max}$ if necessary)  we may assume that $W_4=W_1$.   
 Finally, as noted in the proof of Lemma \ref{newlemma:1}, $\G_{max}\ne\pm\G_{min}$.  
 We may summarize all these working hypotheses as follows:
\begin{equation}\label{W1leastplus}
	\text{SL}(W_1)\le\min\{\text{SL}(W_i)\}\, ,\\ \G_{min}\neq\pm\G_{max},\,
	W_3\neq W_4=W_1\neq W_2
\end{equation}

\begin{rmk}\label{SminReduced}
 Assume the hypotheses \eqref{W1leastplus}.
Let $\b,\b'$ denote respectively the $G$-min up-down connection in Lemma \ref{newlemma:1}
and its neighbouring $G$-max up-down connection.  Consider the part of the top label of $\G$ consisting of the  $\cS$-arcs
connected to $\b$ and $\b'$.  
This is a strongly reduced path containing, for each $s\in\cS$, the $s_{min}$ edge of $\partial\b$ and the $s_{max}$ edge of $\partial\b'$.

It follows that the $s$-edges in $U_{max}\cdot U_{min}$ corresponding to the $s_{max}$ edge of $\partial\b'$ (in $U_{max}$) and the $s$-{min}-edge of $\partial\b$ (in $U_{min}$) survive after strong reduction.

Hence also, whenever in a reduced rectangular or annular picture  a $G$-max up-down connection $\b'$ is
joined to a $G$-min up-down connection $\b$ by all of the $\cS\cup\cT$ arcs in $W_1$, 
after strong reduction the $s_{min}$ arcs from $\b$ and the $s_{max}$ arcs from $\b'$ still go to the top boundary.
\end{rmk}

\section{Completion of proof of Theorem \ref{main}}\label{finalcurtain}

In this final section we complete the proof of Theorem \ref{main}.  By the results in previous sections we are reduced to the case where $Q=W_1$.
Let us suppose that the intersection 
of  $W_1\cdot\pi_1(X\cup\cT,v_1)\cdot W_1^{-1}$ and $\pi_1(X\cup\cS,v_0)$ in $\pi_1(Y,v_0)$ is not cyclic.

By the discussions in \S \ref{formW1}  we may assume the set of hypotheses \eqref{W1leastplus}.
It follows that any reduced rectangular picture $P$ with vertical sides labelled $W_1$   has the form 
\begin{equation}\label{W4form}
\pm(\G_{min}+\D+\G_{max})
\end{equation}
for some $\D$.

Now if $P_1,P_2$ are any two reduced rectangular pictures of the form   $P_j=\G_{min}+\D_j+\G_{max}$, 
then $P_1+P_2$ is also reduced, and by Remark \ref{SminReduced} after strong reduction of the boundary, 
 for each $s\in\cS$
the $s_{max}$ arc on the right hand  $\G_{max}$ disc of $P_1$ still ends on the top, 
as does the $s_{min}$ arc on the left hand $\G_{min}$ disc of $P_2$.
Writing  
 $\ell(P)$ for the number of $\cS$ arcs meeting the top of $P$ between the   left-most $s_{min}$ edge ($s\in\cS$) of
the left hand  $\G_{min}$ disc of $P$ and 
the  right-most  $s_{max}$ edge ($s\in\cS$) of the right hand  $\G_{max}$ disc of $P$,
we have that 
we have $\ell(P_1+P_2) \ge \ell(P_1) + \ell(P_2)$.

We may choose a $(Y,W_1)$-CCE $(U_1,U_2,V_1,V_2)$ and reduced rectangular pictures $P_j$ representing the nullhomotopies $U_j\sim W_1V_jW_1^{-1}$ ($j=1,2$) such that 
\begin{itemize} 
\item $\ell(P_1)$ is smallest possible for $U_1\in K\setminus \{1\}$; and
\item $\ell(P_2)$ is smallest possible for $U_2\in K\setminus\<U_1\>$.
\end{itemize}

Now consider a reduced rectangular picture $P_3$ with strongly reduced top label $U_3$ that is obtained from $P_1-P_2$ by the processes of $\a$-disc cancellation and boundary surgery.  
As before,  $P_3$ has the form  $\pm(\G_{min}+\G_3+\G_{max})$. 
Since  $U_3=U_1U_2^{-1}\notin \<U_1\>$ we have $$\ell(P_3)\ge \ell(P_2)\ge\ell(P_1)>0$$ by the choice of $P_1$ and $P_2$.
But one of $P_3+P_2$, $-P_3+P_1$, after strong reduction on the boundary, 
is a reduced rectangular picture with top label $U_1$ or $U_2$, giving a contradictory inequality 
$\ell(P_1)\ge \ell(P_3)+\ell(P_2)$ or $\ell(P_2)\ge\ell(P_3)+\ell(P_1)$ respectively.  
This contradiction completes the proof.

\end{document}